\documentclass[11pt,letterpaper,reqno]{amsart}

\usepackage[T1]{fontenc}

\usepackage{amsmath}
\usepackage{amssymb}
\usepackage{amsfonts}
\usepackage{amsthm}
\usepackage{bbm}
\usepackage{enumitem}
\usepackage{booktabs}

\usepackage{graphicx}
\usepackage{float}
\usepackage{tikz}
\usepackage{pgfplots}
\pgfplotsset{compat=1.18}
\usetikzlibrary{positioning,shapes.geometric,arrows.meta}

\usepackage{xcolor}
\definecolor{darkgreen}{HTML}{006400}
\usepackage[
colorlinks=true,
linkcolor=blue,
citecolor=green,
urlcolor=magenta,
pdfstartview=FitH
]{hyperref}
\usepackage{doi}
\usepackage{url}

\newtheorem{thm}{Theorem}[section]
\newtheorem{lem}[thm]{Lemma}
\newtheorem{prop}[thm]{Proposition}
\newtheorem{cor}[thm]{Corollary}

\newtheorem{prob}[thm]{Problem}

\theoremstyle{definition}
\newtheorem{eg}[thm]{Example}
\newtheorem{rem}[thm]{Remark}
\newtheorem{defin}[thm]{Definition}
\numberwithin{equation}{section}

\newcommand{\sa}{\mathrm{sa}}
\newcommand{\RR}{\mathbb R}
\newcommand{\CC}{\mathbb C}

\newcommand{\one}{1}

\newcommand{\cl}{\operatorname{cl}}
\DeclareMathOperator{\spanop}{span}
\DeclareMathOperator{\dist}{dist}

\DeclareMathOperator{\dens}{dens}
\DeclareMathOperator{\supp}{supp}

\title[Kadison's problem and central density]
{The solution to Kadison's problem on orthonormal bases of unitaries for type $\mathrm{II}_1$ factors}

\author[Y.~He]{Yixin He}
\address{School of Mathematical Sciences, Fudan University, Shanghai 200433, P.~R.~China}
\email{yixin.he717@gmail.com}

\author[Q.~Tang]{Quanyu Tang}
\address{School of Mathematical Sciences, University of Science and Technology of China, Hefei 230026, P.~R.~China}
\email{tangquanyu827@gmail.com}

\author[Z.~Xu]{Zongben Xu}

\address{School of Mathematics and Statistics, Xi'an Jiaotong University, Xi'an 710049, P.~R.~China}
\email{zbxu@mail.xjtu.edu.cn}

\author[T.~Zhang]{Teng Zhang}
\address{School of Mathematics and Statistics, Xi'an Jiaotong University, Xi'an 710049, P.~R.~China}
\email{teng.zhang@stu.xjtu.edu.cn}

\dedicatory{Dedicated to the memory of Richard V. Kadison (1925--2018)}

		

\begin{document}

\begin{abstract}
	In 1967, Kadison asked whether every type $\mathrm{II}_1$ factor has
	an orthonormal basis, with respect to its trace, consisting of
	unitaries.  Combining the separable case recently established by He--Tang--Zhang
	with the nonseparable case treated here, we prove the following
	characterization.  Let $M$ be a diffuse finite von Neumann algebra
	with a faithful normal tracial state $\tau$, and put
	$\kappa=\dens L^2(M,\tau)$.  Then there exists an orthonormal family $
	B\subset\mathcal S_0(M)
	=
	\{s\in M:s=s^*=s^{-1},\ \tau(s)=0\}$
	such that
	$\{\widehat{\one}\}\cup\{\widehat{s}:s\in B\}$ is an orthonormal
	basis of $L^2(M,\tau)$ if and only if
$\dens L^2\bigl(zM,\tau(z)^{-1}\tau|_{zM}\bigr)=\kappa$
	for every nonzero central projection $z\in Z(M)$.  As an immediate
	corollary, we resolve Kadison's problem for arbitrary type
	$\mathrm{II}_1$ factors.  We also provide a Lean~4 formalization
	of the main results.
\end{abstract}
	
	\maketitle
	
	\begingroup
	\hypersetup{linkcolor=darkgreen}
	\tableofcontents
	\endgroup
	
	\section{Introduction}\label{sec:introduction}
	Let \(M\) be a finite von Neumann algebra equipped with a faithful normal
tracial state \(\tau\).    We denote by
\(L^2(M,\tau)\) the Hilbert
space completion of \(M\) for the norm \(\|x\|_2=\tau(x^*x)^{1/2}\). For
\(x\in M\), its image in \(L^2(M,\tau)\) will be denoted by \(\widehat{x}\). For a normed space $\mathcal X$, we write $\dens \mathcal X$ for its
density character, that is, the least cardinality of a norm-dense
subset of $\mathcal X$. We write
$
M_{\mathrm{sa}}=\{x\in M:x=x^*\}
$
for the self-adjoint part of \(M\).  We denote by
$\mathcal U(M)=\{u\in M:u^*u=uu^*=\one\}$ the unitary group of $M$. We also denote by
\(L^2(M,\tau)_{\mathrm{sa}}\) the real Hilbert space obtained as the
\(\|\cdot\|_2\)-closure of \(M_{\mathrm{sa}}\) in \(L^2(M,\tau)\). 
We use the convention that the inner product is linear in the first variable:
\[
\langle \widehat{x},\widehat{y}\rangle_2=\tau(y^*x),
\qquad x,y\in M.
\]
When no confusion is possible, we identify an element \(x\in M\) with its
vector \(\widehat x\in L^2(M,\tau)\). \(M\) acts on
\(L^2(M,\tau)\) by left multiplication.  A vector \(\xi\in L^2(M,\tau)\) is called a \emph{trace vector} for this
action if \(\langle x\xi,\xi\rangle_2=\tau(x)\) for every \(x\in M\).
We shall write
\[
H_0(M)
=
L^2(M,\tau)_{\sa}\ominus \RR\one
=
\{\xi\in L^2(M,\tau)_{\sa}:\langle \xi,\one\rangle_2=0\}.
\]
Thus, for $x\in M_{\mathrm{sa}}$, we have
$\widehat{x}\in H_0(M)$ if and only if $\tau(x)=0$.  We also write
\[
\mathcal S_0(M)
=
\{s\in M:s=s^*=s^{-1},\ \tau(s)=0\}
\]
for the set of trace-zero symmetries in $M$. The center of $M$ is
denoted by $Z(M)$, and a projection
$z\in Z(M)$ is called a \emph{central projection}.  For every nonzero central projection $z\in Z(M)$, we equip $zM$ with the normalized trace $\tau_z:=\tau(z)^{-1}\tau|_{zM}$.
For a subset
$X\subset M$, we write $W^*(X)$ for the von Neumann subalgebra of $M$
generated by $X$.  Unless a corner is explicitly specified, von Neumann subalgebras
are understood to contain the identity of the ambient algebra.
The algebras $W^*(X)$ and $C^*(X)$ are taken to be unital.

		In 1967, at the Baton Rouge conference, Kadison~\cite{Kad67} posed the
	following problem concerning orthonormal bases of trace vectors in
	type \(\mathrm{II}_1\) factors; see also
	\cite[p.~622, Problem~13]{Ge03} and \cite[Problem~S.2]{Pet20}.
	
	\begin{prob}[Kadison]\label{prob:kadison}
		Let \(M\) be a type \(\mathrm{II}_1\) factor with a faithful normal tracial
		state \(\tau\).  Does \(L^2(M,\tau)\) admit a
		Hilbert-space orthonormal basis consisting of trace vectors?
	\end{prob}
	
	In the tracial standard representation, Problem~\ref{prob:kadison} is also naturally reformulated as a
	\emph{unitary basis problem}.  Indeed, if \(u\in\mathcal U(M)\), then
	\[
	\langle x\widehat u,\widehat u\rangle_2
	=
	\tau(u^*xu)
	=
	\tau(x),\qquad x\in M,
	\]
	so \(\widehat u\) is a trace vector.  Conversely, as recalled in
	Lemma~\ref{lem:trace-vectors-unitaries}, every trace vector in
	\(L^2(M,\tau)\) is represented by a unitary in \(M\).  Thus,
	Problem~\ref{prob:kadison} is equivalent to asking whether
	\(L^2(M,\tau)\) has a Hilbert-space orthonormal basis consisting of unitaries in \(M\).
	
The separable case of Problem~\ref{prob:kadison} has been studied
for nearly 60 years.
 Ching \cite[\S~3, pp.~157--158, definition and Remarks~1 and~3]{Chi73} introduced the notion of an ``ortho-unitary basis,'' with additional
 closure conditions under multiplication and adjoints up to scalar
 factors of modulus one. He observed that the canonical group
 unitaries provide such a basis for a countable discrete group von
 Neumann algebra and proved permanence under free products.  Popa \cite[Corollary~2.6]{Pop85} proved that every separable type
	$\mathrm{II}_1$ factor containing a Cartan subalgebra has an orthonormal basis
	of $L^2(M,\tau)$ consisting of normalizing unitaries.  Choda
	\cite[Theorem~3.4]{Cho87} showed that an $n$-unitary shift of
	the hyperfinite type $\mathrm{II}_1$ factor yields an explicit
	orthonormal unitary basis.  For finite-dimensional von Neumann algebras and the related theory of
	unitary Pimsner--Popa bases for finite-index regular subfactors, see
	Crann, Kribs, and Pereira~\cite[Theorems~2.1, 2.4, and~3.1]{CKP23}. In 2023, De and Mukherjee \cite[Theorem~3.2]{DM23} proved
	that every separable diffuse finite von Neumann algebra admits a uniformly
	bounded self-adjoint orthonormal basis in its GNS space; their construction does not yield
	unitaries.
	The separable case of Problem~\ref{prob:kadison} was finally settled by He, Tang, and Zhang in \cite[Theorem~1]{HTZ26}. Their method depends on the
	noncommutative Lyapunov theorem of Akemann and Weaver
	\cite[Theorem~3.1]{AW03}.  A recent related use of Lyapunov convexity
	in the abelian setting appears in
	\cite[Lemma~2.2]{ACFI26}.

In the nonseparable setting, the difficulty is to preserve bounded
operator representatives through a transfinite extension procedure.
An arbitrary $L^2$-orthogonal projection onto the closed span of previously
chosen symmetries need not preserve $M_{\sa}$. The completed blocks and
bounded backgrounds introduced below address this obstruction.
	
In this paper, our main theorem below gives a complete
answer: for an arbitrary diffuse finite tracial von Neumann algebra,
the existence of a unitary orthonormal basis is characterized exactly
by the $L^2$-densities of its nonzero central summands.  Moreover,
whenever such a basis exists, it can be chosen, apart from the identity,
to consist entirely of trace-zero symmetries.
		\begin{thm}\label{thm:nonseparable-main}
		Let $M$ be a diffuse finite von Neumann algebra with a faithful normal
		tracial state $\tau$, and set $\kappa=\dens L^2(M,\tau)$. Then the
		following conditions are equivalent:
		\begin{enumerate}
			\item[(i)] $L^2(M,\tau)$ admits a Hilbert-space orthonormal basis consisting
			of unitaries in $M$;
			\item[(ii)] $L^2(M,\tau)$ admits a Hilbert-space orthonormal basis consisting
			of self-adjoint unitaries in $M$;
			\item[(iii)] there is an orthonormal family $B\subset\mathcal S_0(M)$ such that
			$\{\widehat{\one}\}\cup\{\widehat{s}:s\in B\}$
			is a Hilbert-space orthonormal basis of $L^2(M,\tau)$;
			\item[(iv)] $(M,\tau)$ satisfies the central density condition: 
				\begin{equation}\tag{CD}\label{eq:CD}
				\dens L^2(zM, \tau_{z})=\kappa
				\qquad\text{for every nonzero central projection }z\in Z(M).
			\end{equation}
		\end{enumerate}
		Equivalently, in the tracial standard representation, these conditions
		are also equivalent to the existence of a trace-vector orthonormal basis.
	\end{thm}
		If $M$ is a type $\mathrm{II}_1$ factor, then $Z(M)=\CC\one$, so
	\eqref{eq:CD} is automatic.  Consequently, Theorem~\ref{thm:nonseparable-main}
	completely resolves Problem~\ref{prob:kadison}.
	\begin{cor}\label{cor}
		Problem \ref{prob:kadison} has a positive answer.
	\end{cor}
	In the abelian case, \eqref{eq:CD} is equivalent to homogeneity of the
	associated atomless probability algebra, whose Maharam type agrees with
	its $L^2$-density. Maharam's classification
	\cite[p.~109, Theorem~1]{Mah42} then represents this algebra as a product
	probability algebra. Finite products of coordinate Walsh functions
	\cite[pp.~8--10]{Wal23} form an orthonormal basis consisting of the identity
	and trace-zero symmetries. Thus the abelian implication
	{\rm (iv)}$\Rightarrow${\rm (iii)} is classical; the issue addressed here
	is its noncommutative extension.

A Lean~4 formalization of Theorem~\ref{thm:nonseparable-main} and Corollary~\ref{cor} is available in~\cite{Zha26a}. A simplified proof of Theorem~\ref{thm:nonseparable-main} is also available in~\cite{Zha26b}. We mention it only briefly, as the present paper develops substantially more of the underlying ideas and contains several further results.

\medskip\noindent\textbf{Sketch of the proof and organization of the paper. }
The proof of Theorem~\ref{thm:nonseparable-main} consists of five main
steps.

\medskip
\noindent
\textbf{Step 1: Reduction to the central density condition.}
Section~\ref{sec:preliminaries} first identifies trace vectors in the
tracial standard representation with vectors represented by unitaries
in $M$. We then reduce Problem~\ref{prob:kadison} to constructing an orthonormal
basis of the real Hilbert space $H_0(M)$
consisting of trace-zero symmetries.

The first genuine obstruction in the nonseparable setting comes from
the center. If $L^2(M,\tau)$ has an orthonormal basis of unitaries,
then every nonzero central summand must have the same $L^2$-density as
the ambient algebra:
\[
\dens L^2(zM,\tau_z)
=
\dens L^2(M,\tau)
=
\kappa,
\qquad
0\neq z\in Z(M).
\]
This proves the necessity of the central density condition
\eqref{eq:CD}.

\medskip
\noindent
\textbf{Step 2: Relative norming and finite approximation.}
Assume from now on that $\kappa>\aleph_0$ and that \eqref{eq:CD}
holds. The first consequence is that every nonzero corner has full
density:
\[
\dens L^2(eMe)=\kappa,
\qquad
0\neq e\in\mathcal P(M).
\]
Hence a von Neumann algebra generated by fewer than $\kappa$
constraints cannot fill such a corner. This produces, inside every
nonzero spectral corner, a nonzero bounded self-adjoint perturbation
annihilating a prescribed family of fewer than $\kappa$ normal
constraints.

Applying this perturbation principle to an extreme point argument
gives the relative symmetry norming lemma. More precisely, if
$r=r^*\neq0$ is orthogonal to a small subalgebra and to a small family
of previously imposed constraints, then there is a trace-zero symmetry
$s$ satisfying the same constraints and
\[
\langle r,s\rangle_2
=
\frac{\|r\|_2^2}{\|r\|_\infty}.
\]
This is the nonseparable analogue of the norming step used in
\cite[Lemma~2]{HTZ26}. Iterating it in the greedy procedure of
\cite[Claim~3]{HTZ26} yields Lemma~\ref{lem:finite-approximation}:
every admissible bounded self-adjoint residual can be approximated,
to arbitrary $L^2$-accuracy, by a finite real linear combination of
new mutually orthogonal trace-zero symmetries.

\medskip
\noindent
\textbf{Step 3: Completed blocks and bounded backgrounds.}
The next difficulty is that, for an arbitrary subfamily of already
chosen symmetries, the corresponding $L^2$-orthogonal projection need
not send $M_{\sa}$ back into $M_{\sa}$. To keep all residuals bounded,
Section~\ref{sec:backgrounds} organizes the construction into
\emph{completed blocks}
\[
L^2(Q)_{\sa}
=
L^2(A)_{\sa}
\oplus
\overline{\spanop_{\RR}}T.
\]
For such a block, the projection onto
$\overline{\spanop_{\RR}}T$ is represented by conditional expectations:
\[
P_Tx
=
E_Q(x)-E_A(E_Q(x)),
\qquad x\in M_{\sa}.
\]
Thus the relevant orthogonal projections preserve bounded
self-adjoint elements.

To pass to small subfamilies without losing this structure, we
introduce bounded backgrounds together with a closure operation
\[
S\longmapsto \cl_\Omega(S).
\]
The closure is generated by
$(\mathbb Q+i\mathbb Q)$-coefficient $*$-polynomials, conditional
expectations, and the countable supports of the corresponding Fourier
coefficients. Proposition~\ref{prop:OA-closure} shows that
\[
|\cl_\Omega(S)|
\leq
\max\{|S|,\aleph_0\},
\]
that the closure is continuous along directed increasing unions, and
that every closed subfamily remains a bounded background. This is the
mechanism that preserves both boundedness and cardinal control.

\medskip
\noindent
\textbf{Step 4: Relative basis extension below the ambient density.}
Section~\ref{sec:extension} first treats the countable case. Given a
self-adjoint polynomial $y$ in the parameters constructed so far, we
use the orthogonal decomposition
\[
y
=
E_N(y)+P_\Omega(y)+b(y),
\qquad
b(y):=y-E_N(y)-P_\Omega(y).
\]
The first term is absorbed into a countably generated subalgebra of
$N$, the second is supported on a countable closed subbackground of
$\Omega$, and the residual $b(y)$ is approximated by new orthogonal
symmetries using Lemma~\ref{lem:finite-approximation}. A countable
task schedule ensures that every polynomial is treated at every
required accuracy. The result is Lemma~\ref{lem:countable-layer}.

The countable construction is then upgraded by transfinite induction
on every infinite cardinal
\[
\aleph_0\leq\lambda<\kappa.
\]
For the inductive step $\lambda>\aleph_0$, at stage $\alpha<\lambda$
the symmetry families in the current completed block have
cardinality, and its von Neumann algebras have $L^2$-density, at most
\[
\lambda_\alpha
=
\max\{\aleph_0,|\alpha|\}
<
\lambda,
\]
so the induction hypothesis can be applied at the smaller cardinal
$\lambda_\alpha$. At limit stages we take increasing unions; the
directed continuity of the background closure and the uniform estimate
by $\max\{\aleph_0,|\alpha|\}$ keep the construction below $\lambda$.
No regularity assumption on $\lambda$ is needed.

This proves Proposition~\ref{prop:small-layer} and yields the relative
basis-extension principle: if $H_0(N)$ already has a symmetry basis and
$X\subset M_{\sa}$ has cardinality below $\kappa$, then one can enlarge
$N$ to a von Neumann algebra $P\supset N\cup X$, still of controlled
$L^2$-density, so that the old symmetry basis extends to an orthonormal
basis of $H_0(P)$.

\medskip
\noindent
\textbf{Step 5: Transfinite completion of the basis.}
Section~\ref{sec:main-proof} includes an explicit Walsh basis in
a separable diffuse abelian subalgebra. The nonseparable completion
starts from the scalar algebra $N_0=\CC\one$ and the empty basis
of $H_0(N_0)=\{0\}$.
Choose a self-adjoint generating set $G\subset M$ with
$|G|\leq\kappa$, and schedule its elements along the initial ordinal
$\theta$ of cardinality $\kappa$.
We construct increasing algebras $N_\alpha$ and nested symmetry
bases $B_\alpha$ for $\alpha\leq\theta$ satisfying
\[
B_\alpha\text{ is an orthonormal basis of }H_0(N_\alpha),
\qquad
\dens L^2(N_\alpha)\leq\max\{\aleph_0,|\alpha|\}.
\]
At each successor, Proposition~\ref{prop:relative-extension}
absorbs the next generator while retaining the old basis.
At a nonzero limit $\delta$, we take the generated union of the
algebras and the union of their bases. All earlier density bounds
are bounded by the same cardinal $\max\{\aleph_0,|\delta|\}$,
which controls the limit even when $\kappa$ is singular.

At the endpoint every element of $G$ has been absorbed, so
$N_\theta=W^*(G)=M$ and $B_\theta$ is a basis of $H_0(M)$.
Lemma~\ref{lem:trace-zero-reduction} then yields the orthonormal
basis
$\{\widehat{\one}\}\cup\{\widehat{s}:s\in B_\theta\}$
of $L^2(M,\tau)$.
Together with Proposition~\ref{prop:central-density-necessary},
this proves Theorem~\ref{thm:nonseparable-main}.
In the factor case the central density condition is automatic,
which gives the affirmative solution of Kadison's problem.

\medskip\noindent\textbf{Acknowledgments and AI tools disclosure. }
	We are grateful to Sorin Popa for helpful comments.
	
	Teng Zhang is supported by the China Scholarship Council, the Young Elite
	Scientists Sponsorship Program for PhD Students (China Association for
	Science and Technology), and the Fundamental Research Funds for the Central
	Universities at Xi'an Jiaotong University (Grant No.~xzy022024045).
	
ChatGPT 5.6 was used for English-language editing, proofreading,
 and as an exploratory tool in discussions
	concerning Proposition~\ref{prop:OA-closure},
	Lemma~\ref{lem:countable-layer}, and
	Proposition~\ref{prop:small-layer}.  In an earlier version of this
	work, the closure argument was formulated using
	Skolem hulls and Skolem functions, and the main result concerned the
	nonseparable type $\mathrm{II}_1$ factor case.  We subsequently
	recast the closure and transfinite-extension arguments entirely in
	operator-algebraic terms.  In doing so, we observed that the method
	extends further: after isolating the central obstruction, it yields the
	equivalence characterization for diffuse finite tracial von Neumann
	algebras proved in the present paper.  	The authors take full
	responsibility for all statements, citations, and proofs in the final
	manuscript.

The Lean formalization in~\cite{Zha26a} was generated using OpenAI's Codex.

	\section{Tracial reductions and density facts}\label{sec:preliminaries}

	We begin with four elementary facts about $L^2$-density and coefficient
	supports.  For completeness and to avoid possible confusion, we include the details.

\begin{lem}\label{lem:density-facts}
	Let $(M,\tau)$ be a finite tracial von Neumann algebra.
	\begin{enumerate} 
		\item[(i)] If $X$ is a subset of $M$, then 
		$\dens L^2(W^*(X)) 
		\leq 
		\max\{|X|,\aleph_0\}.$ 
		
		\item[(ii)] If $(M_i)_{i\in I}$ is a nonempty directed increasing family of von 
		Neumann subalgebras of $M$ and $ 
		M_0=W^*\left(\bigcup_i M_i\right),$ 
		then 
		$ 
		L^2(M_0) 
		= 
		\overline{\bigcup_i L^2(M_i)}^{\|\cdot\|_2}. 
		$ 
		
\item[(iii)] If $Q\subset M$ is a von Neumann subalgebra,
$x\in M$, and $\widehat{x}\in L^2(Q)$, then $x\in Q$.
		
		\item[(iv)] If $T\subset L^2(M)$ is an orthonormal family and 
		$\xi\in L^2(M)$, then 
		$ 
		\supp_T(\xi) 
		:= 
		\{t\in T:\langle\xi,t\rangle_2\neq0\} 
		$ 
		is countable. 
	\end{enumerate} 
\end{lem} 

\begin{proof} 
	For (i), set $\lambda=\max\{|X|,\aleph_0\}$.  Let $\mathcal A_0$ be the 
	unital $*$-algebra over $\mathbb Q+i\mathbb Q$ generated by $X$.  Since 
	every element of $\mathcal A_0$ is a finite linear combination of finite 
	$*$-words in elements of $X$, we have $|\mathcal A_0|\leq\lambda$. 
	Consequently, $C^*(X)$ has a norm-dense subset of cardinality at most 
	$\lambda$.  Its closed unit ball 
	\[ 
	C^*(X)_1=\{a\in C^*(X):\|a\|_\infty\leq1\} 
	\] 
	therefore has a norm-dense subset $D\subset C^*(X)_1$ with 
	$|D|\leq\lambda$.  By Kaplansky's density theorem 
	\cite[p.~329, Theorem~5.3.5]{KR97a}, the ball $C^*(X)_1$ is strongly 
	dense in the unit ball of $W^*(X)=C^*(X)''$.  We use the standard 
	representation on $L^2(M)$.  If a bounded net $(a_j)$ converges strongly 
	to $a$, then, on applying the operators to the canonical vector 
	$\widehat{\one}$, 
	\[ 
	\|a_j-a\|_2 
	=\|(a_j-a)\widehat{\one}\|_2 
	\longrightarrow0. 
	\] 
	Hence the unit ball of $C^*(X)$ is $L^2$-dense in the unit ball of 
	$W^*(X)$.  Since $\|a-b\|_2\leq\|a-b\|_\infty$, the set $D$ is 
	$L^2$-dense in the latter ball.  Consequently, 
	$\bigcup_{n\geq1}nD$ is $L^2$-dense in $W^*(X)$.  Because $W^*(X)$ is 
	$L^2$-dense in $L^2(W^*(X))$ and 
	$\left|\bigcup_{n\geq1}nD\right|\leq\lambda$, we conclude that 
	\[ 
	\dens L^2(W^*(X))\leq\lambda. 
	\] 
	This proves (i). 
	
	For (ii), let $ 
	A_\infty=\overline{\bigcup_{i\in I}M_i}^{\|\cdot\|}.$ 
	Directedness makes $\bigcup_{i\in I}M_i$ a unital $*$-algebra.  Thus 
	$A_\infty$ is a $C^*$-subalgebra of $M$, and 
	\[ 
	A_\infty''=W^*\left(\bigcup_{i\in I}M_i\right)=M_0. 
	\] 
	By Kaplansky's density 
	theorem \cite[p.~329, Theorem~5.3.5]{KR97a}, the unit ball of 
	$A_\infty$ is strongly dense in the unit ball of $M_0$.  As above, 
	strong convergence in the standard representation implies convergence in 
	$\|\cdot\|_2$.  Thus $A_\infty$ is $L^2$-dense in $M_0$, while 
	$\bigcup_{i\in I}M_i$ is norm, and hence $L^2$, dense in $A_\infty$. 
	It follows that 
	\[ 
	L^2(M_0) 
	=\overline{\bigcup_{i\in I}M_i}^{\|\cdot\|_2} 
	=\overline{\bigcup_{i\in I}L^2(M_i)}^{\|\cdot\|_2}. 
	\] 
	This proves (ii). 
	
For (iii), since the modular automorphism group of $\tau$ is trivial,
Takesaki's conditional-expectation theorem
\cite[Theorem~3]{Tak72} yields a trace-preserving normal conditional
expectation $E_Q:M\to Q$.  We claim that its $L^2$-extension is the
orthogonal projection from $L^2(M)$ onto $L^2(Q)$; the argument also
proves uniqueness.  Indeed, for $x\in M$ and $y\in Q$, trace preservation 
and $Q$-bimodularity give 
\begin{align*} 
	\left\langle\widehat{x-E_Q(x)},\widehat y\right\rangle_2 
	&=\tau\bigl(y^*(x-E_Q(x))\bigr)\\ 
	&=\tau\!\left(E_Q\bigl(y^*(x-E_Q(x))\bigr)\right)\\ 
	&=\tau\bigl(y^*E_Q(x-E_Q(x))\bigr) 
	=0. 
\end{align*} 
Thus $\widehat{x-E_Q(x)}\perp L^2(Q)$ and 
$\widehat{E_Q(x)}\in L^2(Q)$, proving the claim.  If 
$\widehat x\in L^2(Q)$, the projection fixes $\widehat x$, so 
$\|x-E_Q(x)\|_2=0$.  Faithfulness of $\tau$ gives 
$x=E_Q(x)\in Q$, which proves (iii). 

We shall also use the scalar characterization
\[
E_Q(x)=E_Q(y)
\quad\Longleftrightarrow\quad
\tau\bigl(a(x-y)\bigr)=0\quad(a\in Q).
\]
The forward implication follows from trace preservation and
bimodularity; for the reverse implication, take
$a=E_Q(x-y)^*$ and use faithfulness.

	For (iv), for each $n\geq1$ put $ 
	T_n=\left\{t\in T:|\langle\xi,t\rangle_2|\geq\frac1n\right\}.$ 
	If $F\subset T_n$ is finite, Bessel's inequality gives 
	\[ 
	\frac{|F|}{n^2} 
	\leq\sum_{t\in F}|\langle\xi,t\rangle_2|^2 
	\leq\|\xi\|_2^2. 
	\] 
Hence $T_n$ is finite. If $t\in\supp_T(\xi)$, then
$|\langle\xi,t\rangle_2|>0$, so there exists $m\geq1$ such that
\[
|\langle\xi,t\rangle_2|\geq\frac1m.
\]
Thus $t\in T_m$, and therefore
\[
\supp_T(\xi)=\bigcup_{n\geq1}T_n.
\]
	The right-hand side is countable, proving (iv). 
\end{proof}

	We next identify the trace vectors for the left action in the tracial standard
	representation.
	
	\begin{lem}\label{lem:trace-vectors-unitaries}
		Let \(M\) be a finite von Neumann algebra equipped with a faithful normal
	tracial state \(\tau\).  A vector \(\xi\in L^2(M,\tau)\) is a trace
		vector for \(M\) if and only if
		$\xi=\widehat{u}$
		for some unitary \(u\in\mathcal U(M)\).
	\end{lem}
	
	\begin{proof}
		If \(u\in\mathcal U(M)\), then for every \(x\in M\),
		\[
		\langle x\widehat{u},\widehat{u}\rangle_2
		=\tau(u^*xu)
		=\tau(x),
		\]
		so \(\widehat{u}\) is a trace vector.
		
		Conversely, suppose that \(\xi\) is a trace vector, and let
		\(\Omega=\widehat{1}\). Since \(\Omega\) is a cyclic trace vector for
		\(M\), the classical reflection theorem
		\cite[p.~472, Theorem~7.2.15]{KR97b} yields a \(*\)-anti-isomorphism $
		\Theta:M\longrightarrow M'$
		characterized by
		\[
		\Theta(a)\Omega=a\Omega,\qquad a\in M,
		\]
		and, moreover, \(M'\) is finite.
		
		Define
		\[
		V_0:M\Omega\longrightarrow L^2(M,\tau),
		\qquad
		V_0(x\Omega)=x\xi.
		\]
		Since \(\xi\) is a trace vector, for every \(x\in M\),
		\[
		\|V_0(x\Omega)\|_2^2
		=\langle x^*x\xi,\xi\rangle_2
		=\tau(x^*x)
		=\|x\Omega\|_2^2.
		\]
		Thus \(V_0\) extends to an isometry
		\(V\in B(L^2(M,\tau))\). Moreover, for \(a,x\in M\),
		\[
		V(ax\Omega)=ax\xi=aV(x\Omega),
		\]
		and hence \(V\in M'\). Since \(M'\) is finite, every isometry in
		\(M'\) is unitary. Therefore \(V\in\mathcal U(M')\).
		
		By the surjectivity of \(\Theta\), there exists \(u\in M\) such that
		\(V=\Theta(u)\). Since \(\Theta\) is a \(*\)-anti-isomorphism and \(V\)
		is unitary, \(u\in\mathcal U(M)\). Finally,
		\[
		\xi
		=V\Omega
		=\Theta(u)\Omega
		=u\Omega
		=\widehat{u}.
		\]
		This proves the lemma.
	\end{proof}
	
	The following lemma reduces the complex trace-vector basis problem to a real
	symmetry-basis problem.
	
	\begin{lem}\label{lem:trace-zero-reduction}
		Let $M$ be a finite von Neumann algebra with a faithful normal tracial
		state $\tau$. If $B\subset\mathcal S_0(M)$ is a Hilbert-space orthonormal
		basis of the real Hilbert space $H_0(M)$, then
		$\{\widehat{\one}\}\cup\{\widehat{s}:s\in B\}$ is a Hilbert-space orthonormal
		basis of the complex Hilbert space $L^2(M,\tau)$. Moreover, every
		vector in $\{\widehat{\one}\}\cup\{\widehat{s}:s\in B\}$ is a trace
		vector for the left action of $M$.
	\end{lem}
	
	\begin{proof}
		Clearly, for self-adjoint $a,b\in M$, $\tau(ba)$ is real.
		Thus real orthogonality inside $L^2(M,\tau)_{\sa}$ is the same as complex
		Hilbert-space orthogonality for these vectors.  Since
		\(B\) is an orthonormal basis of \(H_0(M)\) and \(B\subset\mathcal S_0(M)\),
		the family \(\{\one\}\cup B\) is orthonormal in \(L^2(M,\tau)\).
		
		Let \(\mathcal L\) be the closed complex linear span of \(\{\one\}\cup B\).
		Since
		\[
		L^2(M,\tau)_{\sa}
		=
		\RR\one\oplus H_0(M)
		=
		\overline{\spanop_{\RR}}\bigl(\{\one\}\cup B\bigr),
		\]
		we have \(L^2(M,\tau)_{\sa}\subset \mathcal L\).  Hence
		\(M_{\sa}\subset\mathcal L\), and therefore
		\[
		M=M_{\sa}+iM_{\sa}\subset\mathcal L.
		\]
		Since \(M\) is dense in \(L^2(M,\tau)\), it follows that
		\(\mathcal L=L^2(M,\tau)\).  Thus \(\{\one\}\cup B\) is an orthonormal
		basis of the complex Hilbert space \(L^2(M,\tau)\).
		
		If $s=s^*=s^{-1}\in M$, then for every $x\in M$,
		\[
		\langle xs,s\rangle_2=\tau(sxs)=\tau(xss)=\tau(x).
		\]
		Thus every self-adjoint unitary is a trace vector, and so is $\one$.
	\end{proof}

	We conclude this section by proving that the central density condition in Theorem~\ref{thm:nonseparable-main} is
	necessary.
	
	\begin{prop}\label{prop:central-density-necessary}
		Let $M$ be a diffuse finite von Neumann algebra with a faithful normal
		tracial state $\tau$, and put $\kappa=\dens L^2(M,\tau)$. If
		$L^2(M,\tau)$ admits a Hilbert-space orthonormal basis consisting of unitaries
		in $M$, then $(M,\tau)$ satisfies \eqref{eq:CD}.
	\end{prop}
	
	\begin{proof}
Let $\{u_i:i\in I\}\subset\mathcal U(M)$ be an orthonormal basis of
$L^2(M,\tau)$. Diffuseness supplies arbitrarily large finite families
of nonzero orthogonal projections, so $L^2(M,\tau)$ is infinite-dimensional
and $|I|=\dens L^2(M,\tau)=\kappa$.

Fix a nonzero central projection $z$ and put $K=zL^2(M,\tau)$.
The map on bounded vectors
\[
L^2(zM,\tau_z)\longrightarrow K,
\qquad \widehat{x}^{\,z}\longmapsto\tau(z)^{-1/2}\widehat{x},
\]
extends to a surjective isometry: its image contains the dense subspace
$zM\widehat{\one}$ of $K$. Thus $\mu:=\dens K=\dens L^2(zM,\tau_z)$.
Since $zM$ is nonzero and diffuse, $\mu$ is infinite.

Choose a dense set $D\subset K$ with $|D|=\mu$. For each $d\in D$,
Lemma~\ref{lem:density-facts}(iv) shows that
\[
I_d=\{i\in I:\langle d,u_i\rangle_2\neq0\}
\]
is countable. Every index belongs to some $I_d$. Indeed,
\[
\|zu_i\|_2^2=\tau(u_i^*zu_i)=\tau(z)>0,
\qquad
\langle d,u_i\rangle_2=\langle d,zu_i\rangle_2\quad(d\in K).
\]
If the latter pairing vanished on $D$, continuity and density would
make $zu_i$ orthogonal to all of $K$, contradicting $zu_i\neq0$.
Consequently,
\[
I=\bigcup_{d\in D}I_d,
\qquad
\kappa=|I|\leq\mu\cdot\aleph_0=\mu.
\]
The orthogonal projection onto $K$ sends a dense subset of $L^2(M)$
to a dense subset of $K$, so $\mu\leq\kappa$. Hence
$\dens L^2(zM,\tau_z)=\kappa$, as required. No regularity assumption
on either density cardinal is used.
\end{proof}
	
	\section{The relative norming lemma}\label{sec:norming}
	
Whenever $(M,\tau)$ is a finite tracial von Neumann algebra and
$N\subset M$ is a von Neumann subalgebra, let
$E_N:M\to N$ denote the unique trace-preserving normal conditional
expectation.  For a von Neumann algebra $P$, let $\mathcal P(P)$ denote
its set of projections.  The \emph{central support} $c(e)$ of a
projection $e$ is the least central projection $z$ such that $e\leq z$.
For subsets $X,Y\subset L^2(M)$, the notation $X\perp Y$ means that every
element of $X$ is orthogonal to every element of $Y$; a singleton may be
written without braces.  If $N\subset M$ and $\Gamma\subset M$, then
$E_N(\Gamma)=0$ means that $E_N(\gamma)=0$ for every $\gamma\in\Gamma$.
	
	The first lemma is the full-corner principle at the level of $L^2$-density.
We prove it using the closed two-sided ideal generated by the projection.
	
\begin{lem}\label{lem:corner-density}
Let $(M,\tau)$ be a finite tracial von Neumann algebra satisfying
\eqref{eq:CD}, and set $\kappa=\dens L^2(M,\tau)>\aleph_0$.  If $0\neq e\in\mathcal P(M)$, then
	\[
	\dens L^2(eMe)=\kappa.
	\]
\end{lem}
	
	\begin{proof}
Put $z=c(e)$ and $\mu=\dens L^2(eMe)$. First observe that
\[
K_e:=\overline{\spanop_{\CC}\{\widehat{aeb}:a,b\in M\}}^{\|\cdot\|_2}
=zL^2(M).
\]
Here the algebraic span is the two-sided ideal generated by $e$,
without adjoining the identity. To verify the equality, note that $K_e$
is invariant under left and right multiplication by every element of $M$
and its adjoint. Its orthogonal projection therefore commutes with both
actions. In the tracial standard representation these actions are mutual
commutants, so this projection is left multiplication by a central
projection $w\in M$. Since $\widehat e\in K_e$, we have $we=e$.
Conversely, every central projection $v$ with $e\leq v$ fixes each $aeb$,
and hence $w\leq v$. Thus $w=c(e)=z$.

In particular, $\widehat z\in K_e$. Choose a sequence of finite linear
combinations of the vectors $\widehat{aeb}$ converging to $\widehat z$.
Collecting their finite supports gives countable families $(a_i)$ and
$(b_i)$ such that
\[
\widehat z\in
\overline{\spanop_{\CC}\{\widehat{a_i e b_i}:i\geq1\}}^{\|\cdot\|_2}.
\]
Only one vector is being approximated here; separability of $K_e$ is
not required. Let
\[
\mathcal C=
\overline{\spanop_{\CC}\{\widehat{a_i c b_j}:i,j\geq1,\ c\in eMe\}}
 ^{\|\cdot\|_2}.
\]
We show that $\mathcal C=zL^2(M)$ using two successive applications of
fixed bounded multipliers. If $x\in zM$, fix $i$ and multiply the
approximation of $z$ on the left by $a_i e b_i x$. Each resulting term is
\[
a_i e b_i x a_j e b_j
=a_i\bigl(e b_i x a_j e\bigr)b_j,
\qquad e b_i x a_j e\in eMe.
\]
Left multiplication is $L^2$-continuous, and $xz=x$, so
$\widehat{a_i e b_i x}\in\mathcal C$. Now right multiplication by this
fixed $x$ is also $L^2$-continuous. Applying it to the same approximation
of $z$ gives $\widehat x=\widehat{zx}\in\mathcal C$. The reverse
inclusion follows from $e\leq z$ and the centrality of $z$.
Taking closures proves the claimed equality. This argument requires no
uniform operator-norm bound on the approximating finite sums.

For each $i,j$, the map
$c\mapsto a_i c b_j$ extends continuously from the corner to its
$L^2$-completion, since
\[
\|a_i c b_j\|_2
\leq\|a_i\|_\infty\|b_j\|_\infty\|c\|_2.
\]
Normalizing the trace on $eMe$ merely rescales this bound by
$\tau(e)^{1/2}$. Choose a dense set of cardinality $\mu$ in that
completion. The rational complex span of its images under this
countable family of maps is dense in $\mathcal C$, and has cardinality
at most $\max\{\mu,\aleph_0\}$. Therefore
\[
\dens L^2(zM,\tau_z)
\leq\max\{\mu,\aleph_0\}.
\]
Since $z\neq0$, condition~\eqref{eq:CD} identifies the left-hand side
with $\kappa$. The natural corner embedding gives $\mu\leq\kappa$,
and $\kappa>\aleph_0$ now forces $\mu=\kappa$.
\end{proof}
	
	The next observation is the source of all simultaneous orthogonality
	conditions.  It says that a small von Neumann algebra cannot fill a
	nonzero corner.
	
	\begin{lem}\label{lem:perturbation}
Let $(M,\tau)$ be a finite tracial von Neumann algebra satisfying
\eqref{eq:CD}, and set $\kappa=\dens L^2(M,\tau)>\aleph_0$.	Let $N\subset M$ be a von Neumann subalgebra satisfying
$\dens L^2(N)<\kappa$, let $\Gamma\subset M_{\sa}$ satisfy
$|\Gamma|<\kappa$, and let $0\neq e\in\mathcal P(M)$.  There is a nonzero $h=h^*\in eMe$ such that
		\[
		E_N(h)=0,
		\qquad
		\tau(\gamma h)=0\quad(\gamma\in\Gamma).
		\]
	\end{lem}
	
	\begin{proof}
Choose $D_N\subset N$ whose canonical vectors are dense in $L^2(N)$ and
whose cardinality is at most $\max\{\dens L^2(N),\aleph_0\}$.
Such bounded representatives are obtained by approximating each member of
a dense subset of $L^2(N)$ to accuracy $1/n$, for every $n\geq1$.
Set
\[
\lambda=\max\{\dens L^2(N),|\Gamma|,\aleph_0\}<\kappa,
\qquad D=W^*(D_N\cup\Gamma\cup\{e\}).
\]
The inclusion of $D_N$ makes $L^2(N)\subset L^2(D)$; hence
Lemma~\ref{lem:density-facts}(iii) gives $N\subset D$.
Part~(i) of the same lemma gives $\dens L^2(D)\leq\lambda<\kappa$.
Only a finite maximum of cardinals below $\kappa$ is used here.

If $eMe$ were contained in $D$, its $L^2$-completion, after the fixed
rescaling for the normalized corner trace, would embed isometrically
into $L^2(D)$. This would contradict Lemma~\ref{lem:corner-density}.
Thus there is $x\in eMe\setminus D$. At least one of
$(x+x^*)/2$ and $(x-x^*)/(2i)$ is outside $D$, so choose
$a=a^*\in eMe\setminus D$ and put $h=a-E_D(a)$.
Since $e\in D$, bimodularity gives $E_D(a)\in eDe$.
Therefore $0\neq h=h^*\in eMe$. Moreover $h\perp L^2(D)$, so
$N\subset D$ and $\Gamma\subset D$ imply
$E_N(h)=0$ and $\tau(\gamma h)=0$ for every $\gamma\in\Gamma$.
\end{proof}
	
We now prove a relative version of \cite[Lemma~2]{HTZ26}.
The finite-constraint antecedents are the operator-algebraic Lyapunov
theorems of Akemann and Anderson \cite{AA91} and Akemann and Weaver
\cite[p.~282, Theorem~3.1]{AW03}. In the present setting,
Lemma~\ref{lem:perturbation} allows us to impose the
conditional-expectation constraint together with a family of fewer than
$\kappa$ normal scalar constraints.
	
	\begin{lem}\label{lem:relative-norming} Let $(M,\tau)$ be a finite tracial von Neumann algebra satisfying
		\eqref{eq:CD}, and set $\kappa=\dens L^2(M,\tau)>\aleph_0$.
		Let $N\subset M$ be a von Neumann subalgebra satisfying
		$\dens L^2(N)<\kappa$, let $\Gamma\subset M_{\sa}$ satisfy
		$|\Gamma|<\kappa$, and let
		$0\neq r=r^*\in M$ satisfy
		\[
		E_N(r)=0,
		\qquad
		\tau(\gamma r)=0\quad(\gamma\in\Gamma).
		\]
		Then there is $s\in\mathcal S_0(M)$ such that
		\[
		E_N(s)=0,
		\qquad
		\tau(\gamma s)=0\quad(\gamma\in\Gamma),
		\qquad
		\langle r,s\rangle_2=\frac{\|r\|_2^2}{\|r\|_\infty}.
		\]
	\end{lem}
	
	\begin{proof}
		Put $c=\|r\|_\infty$ and $q=(\one+r/c)/2$.  Consider the ultraweakly
		compact convex set
		\[
		K=\left\{0\leq x\leq\one:
		\begin{array}{l}
			E_N(x)=E_N(q),\\
			\tau(\gamma x)=\tau(\gamma q)\quad(\gamma\in\Gamma),\\
			\tau(rx)=\tau(rq)
		\end{array}\right\}.
		\]
	The order interval $[0,\one]$ is ultraweakly compact. By the scalar
	characterization established in Lemma~\ref{lem:density-facts}(iii),
	the first constraint is equivalent to
	\[
	\tau\bigl(a(x-q)\bigr)=0\qquad(a\in N).
	\]
	Each of these is a closed scalar constraint, because
	$x\mapsto\tau(ax)$ is a normal linear functional. The same applies
	to the tests indexed by $\Gamma$ and to the test defined by $r$.
	Thus all defining constraints are ultraweakly closed, without
	requiring a continuity assertion for the algebra-valued map $E_N$.
	Thus $K$ is a nonempty ultraweakly compact convex set.  By the Krein--Milman theorem
		\cite[p.~142, Theorem~7.4]{Con90}, it has an extreme point $p$.  If $p$ is
		not a projection, we first construct a nonzero corner on which
		$p$ stays away from both endpoints of $[0,\one]$. This
		localization applies to any nonprojection positive contraction.
		Indeed, $\sigma(p)$ contains a point $t\in(0,1)$, since otherwise
		continuous functional calculus would give $p^2=p$. Choose
		$0<\varepsilon<\min\{t,1-t\}$ and put
		\[
		f(s)=\max\{0,\min\{s-\varepsilon,1-\varepsilon-s\}\},
		\qquad b=f(p).
		\]
		Then $b\geq0$, $b\neq0$, and $b$ commutes with $p$. Moreover,
		\[
		b(p-\varepsilon\one)b\geq0,\qquad
		b((1-\varepsilon)\one-p)b\geq0.
		\]
		In a faithful normal representation, let $e$ be the orthogonal
		projection onto $\overline{\operatorname{ran}b}$. This subspace
		reduces both $p$ and $M'$, so $e\in M$, $ep=pe$, and $e\neq0$.
		The two quadratic-form inequalities extend from
		$\operatorname{ran}b$ to its closure and give
		\[
		\varepsilon e\leq epe\leq(1-\varepsilon)e.
		\]
		Notice that $0<\varepsilon<1/2$.

		Apply Lemma~\ref{lem:perturbation} to
		$\Gamma\cup\{r\}$ and this corner. After multiplication by a
		sufficiently small positive scalar, it gives
		$0\neq h=h^*\in eMe$ with $\|h\|_\infty<\varepsilon$ which
		annihilates all the defining constraints of $K$. On $eL^2(M)$,
		\[
		\varepsilon e\leq epe\leq(1-\varepsilon)e,\qquad
		-\varepsilon e\leq h\leq\varepsilon e.
		\]
		On $(\one-e)L^2(M)$, both $p+h$ and $p-h$ agree with $p$.
		Since $e$ commutes with $p$ and $h=ehe$, these inequalities give
		$0\leq p\pm h\leq\one$.
	Therefore $p+h$ and $p-h$ are distinct positive contractions in $K$,
	contradicting the extremality of $p$.  Thus $p$ is a projection.
		
		Set $s=2p-\one$.  The equalities defining $K$, together with
		$E_N(r)=0$ and $\tau(\gamma r)=0$, give
		$E_N(s)=0$ and $\tau(\gamma s)=0$.  In particular, $\tau(s)=0$ because $E_N$ is trace-preserving.  Finally,
		\begin{equation}\label{eq:norming-identity}
			\tau(rs)=2\tau(rp)-\tau(r)
			=2\tau(rq)-\tau(r)
			=\frac{\tau(r^2)}{c}.
		\end{equation}
		Hence $s$ has all the required properties.
	\end{proof}
	
	The norming identity \eqref{eq:norming-identity} yields the following finite
	approximation through the greedy construction of
	\cite[Claim~3]{HTZ26}.
	Here $\dist$ denotes distance in the $L^2$-norm.
	
	\begin{lem}\label{lem:finite-approximation}Let $(M,\tau)$ be a finite tracial von Neumann algebra, set
		$\kappa=\dens L^2(M,\tau)$, and assume that
		$\kappa>\aleph_0$ and that $(M,\tau)$ satisfies \eqref{eq:CD}.
		Let $N\subset M$  be a von Neumann subalgebra and $\Gamma\subset\mathcal S_0(M)$ satisfy
		\[
		\dens L^2(N)<\kappa,
		\qquad |\Gamma|<\kappa,
		\qquad E_N(\Gamma)=0.
		\]
		Let $V$ be the real span of finitely many orthonormal symmetries,
		each orthogonal to $L^2(N)_{\sa}$ and to $\Gamma$.  If
		$b=b^*\in M$, $E_N(b)=0$, $b\perp\Gamma$, and $\varepsilon>0$, then
		there are $m\geq0$ and mutually orthogonal symmetries
		$u_1,\ldots,u_m\in\mathcal S_0(M)$, orthogonal to
		$L^2(N)_{\sa}$, $\Gamma$, and $V$, such that
		\[
		\dist\bigl(b,V+\spanop_{\RR}\{u_1,\ldots,u_m\}\bigr)<\varepsilon.
		\]
	\end{lem}
	
	\begin{proof}
		Let $P_V$ denote the real $L^2$-orthogonal projection onto $V$, and put
		$r_0=b-P_Vb$.  Since the chosen symmetry basis of $V$ is orthogonal to
		$L^2(N)_{\sa}$ and to $\Gamma$, we have $E_N(r_0)=0$ and
		$r_0\perp\Gamma$.  Inductively, the same relations hold for every residual,
		and $r_{j-1}$ is orthogonal to $V$ and to the previously chosen $u_i$.
		If $\|r_0\|_2<\varepsilon$, take $m=0$.  Otherwise, whenever
		$\|r_{j-1}\|_2\geq\varepsilon$, apply
		Lemma~\ref{lem:relative-norming} with the constraints
		$\Gamma$, the finitely many symmetries spanning $V$, and the symmetries
		$u_1,\ldots,u_{j-1}$ already chosen.  This gives a new orthogonal symmetry $u_j$ such that
		\begin{equation}\label{eq:greedy-step}
			\alpha_j:=\langle r_{j-1},u_j\rangle_2
			=\frac{\|r_{j-1}\|_2^2}{\|r_{j-1}\|_\infty},
			\qquad
			r_j=r_{j-1}-\alpha_j u_j.
		\end{equation}
		Since $u_j$ is a unit vector and $r_j\perp u_j$, we have
		\[
		\|r_j\|_2^2
		=
		\|r_{j-1}\|_2^2-\alpha_j^2.
		\]
		Iterating this identity yields
		\begin{equation}\label{eq:greedy-energy}
			\sum_{i=1}^j\alpha_i^2
			=
			\|r_0\|_2^2-\|r_j\|_2^2
			\leq \|r_0\|_2^2,
			\qquad j\geq1.
		\end{equation}
		
		Moreover, by \eqref{eq:greedy-step},
		\[
		r_{j-1}
		=
		r_0-\sum_{i<j}\alpha_i u_i.
		\]
		Since each $u_i$ is a symmetry, $\|u_i\|_\infty=1$.  Hence,
		by the triangle inequality, Cauchy--Schwarz inequality, and
		\eqref{eq:greedy-energy},
		\begin{equation}\label{eq:greedy-infty-bound}
			\begin{aligned}
				\|r_{j-1}\|_\infty
				&\leq
				\|r_0\|_\infty+\sum_{i<j}|\alpha_i| \\
				&\leq
				\|r_0\|_\infty
				+\sqrt{j-1}\left(\sum_{i<j}\alpha_i^2\right)^{1/2} \\
				&\leq
				\|r_0\|_\infty+\sqrt{j-1}\,\|r_0\|_2.
			\end{aligned}
		\end{equation}
		
		Suppose that the process never stops.  Then
		$\|r_{j-1}\|_2\geq\varepsilon$ for every $j\geq1$.  Combining
		\eqref{eq:greedy-step} with \eqref{eq:greedy-infty-bound}, we obtain
		\begin{equation}\label{eq:greedy-lower-bound}
			\alpha_j^2
			=
			\frac{\|r_{j-1}\|_2^4}{\|r_{j-1}\|_\infty^2}
			\geq
			\frac{\varepsilon^4}
			{\bigl(\|r_0\|_\infty+\sqrt{j-1}\,\|r_0\|_2\bigr)^2}.
		\end{equation}
	Put $C=\|r_0\|_\infty+\|r_0\|_2>0$. By
	\eqref{eq:greedy-lower-bound}, for $j\geq1$,
\[
\|r_0\|_\infty+\sqrt{j-1}\,\|r_0\|_2\leq C\sqrt j,
\qquad
\alpha_j^2\geq\frac{\varepsilon^4}{C^2j}.
\]
The harmonic series therefore contradicts the uniform energy bound
\eqref{eq:greedy-energy}. Thus the construction must terminate
		after finitely many steps, necessarily at some $m$ for which
		\[
		\|r_m\|_2<\varepsilon.
		\]
	\end{proof}
	
	\section{Bounded backgrounds from conditional expectations}
	\label{sec:backgrounds}
	
	Recall that the $L^2$-extension of $E_N$ is the orthogonal projection onto
	$L^2(N)$, as proved in
	Lemma~\ref{lem:density-facts}(iii).  Hence
	\[
	E_N(x)=0\quad\Longleftrightarrow\quad x\perp L^2(N)
	\]
	for $x\in M$; this equivalence will be used without further comment.
	
\subsection{Completed blocks}

We introduce the notion of a completed block.
	\begin{defin}[Completed block]\label{def:block} 	Let $(M,\tau)$ be a finite tracial von Neumann algebra.
		A \emph{completed block in $M$} is a triple $(A,Q,T)$, where
		$A\subset Q\subset M$ are von Neumann subalgebras and
		$T\subset\mathcal S_0(Q)$ is an orthonormal family,
		such that
		\begin{equation}\label{eq:block}
			L^2(Q)_{\sa}
			=L^2(A)_{\sa}\oplus\overline{\spanop_{\RR}}T.
		\end{equation}
	\end{defin}
	
\begin{rem}
	The decomposition \eqref{eq:block} allows us to express the relevant
	orthogonal projection in terms of conditional expectations.  Let
	$P_T$ denote the real $L^2$-orthogonal projection onto
	$\overline{\spanop_{\RR}}T$.  Then, for every $x=x^*\in M$,
		\begin{equation}\label{eq:block-projection}
			P_Tx
			=
			E_Q(x)-E_A(E_Q(x))
			\in M_{\sa}.
		\end{equation}
		Indeed, the $L^2$-extensions of $E_A$ and $E_Q$ are the orthogonal
		projections onto $L^2(A)$ and $L^2(Q)$, respectively, and
		$E_AE_Q=E_A$. Hence $E_Q-E_AE_Q$ is precisely the orthogonal projection
		onto
		$
		L^2(Q)_{\sa}\ominus L^2(A)_{\sa}
		=
		\overline{\spanop_{\RR}}T,
		$
		which proves \eqref{eq:block-projection}. In particular, the bounded
representative satisfies
\[
\|P_Tx\|_\infty\leq2\|x\|_\infty,
\]
by contractivity of both conditional expectations. The projection
identity itself holds on all of $L^2(M)_{\sa}$ by continuity; boundedness
of its representative is asserted only when $x\in M_{\sa}$.
	\end{rem}

	The following elementary example illustrates Definition~\ref{def:block}.
	
	\begin{eg}\label{eg:matrix-completed-block}
		Let $Q=M_2(\mathbb C)$ with the normalized trace
		$\tau(x)=\frac12\operatorname{Tr}(x),$
		and let
		\[
		X=
		\begin{pmatrix}
			0&1\\
			1&0
		\end{pmatrix},
		\qquad
		Y=
		\begin{pmatrix}
			0&-i\\
			i&0
		\end{pmatrix},
		\qquad
		Z=
		\begin{pmatrix}
			1&0\\
			0&-1
		\end{pmatrix}
		\]
		be the Pauli matrices. Then $X,Y,Z$ are trace-zero symmetries, and
		$\{\one,X,Y,Z\}$
		is an orthonormal basis of the real Hilbert space
$L^2(Q,\tau)_{\sa}$. Indeed, direct multiplication gives
$\tau(V^*W)=\delta_{V,W}$ for $V,W\in\{\one,X,Y,Z\}$, and every
self-adjoint matrix has the real expansion
\[
\begin{pmatrix}a&b\\\overline b&d\end{pmatrix}
=\frac{a+d}{2}\one+(\operatorname{Re}b)X
 -(\operatorname{Im}b)Y+\frac{a-d}{2}Z,
\qquad a,d\in\RR.
\]
The negative sign in the coefficient of $Y$ corresponds to its upper
off-diagonal entry $-i$. In finite dimension this real span is already closed.
		
		The diagonal matrices form a von Neumann subalgebra containing $Z$,
and every diagonal matrix is a complex linear combination of $\one$ and $Z$.
Hence $A:=W^*(Z)=\spanop_{\CC}\{\one,Z\}$.
Put $T=\{X,Y\}$.
		Since
		$L^2(A,\tau)_{\sa}=\spanop_{\mathbb R}\{\one,Z\},$
		we have the orthogonal decomposition
		\[
		L^2(Q,\tau)_{\sa}
		=
		L^2(A,\tau)_{\sa}
		\oplus
		\spanop_{\mathbb R}\{X,Y\}.
		\]
		Thus $\bigl(W^*(Z),M_2(\mathbb C),\{X,Y\}\bigr)$ is a completed block.
	\end{eg}

\subsection{Bounded backgrounds}

For an orthonormal family $\Omega\subset L^2(M)_{\sa}$, let $P_\Omega$
denote the real $L^2$-orthogonal projection onto
$\overline{\spanop_{\RR}}\Omega$.  We now extend
\eqref{eq:block-projection} beyond a single completed block.  This
motivates the following definition.

\begin{defin}[Bounded background]\label{def:background}
	Let $(M,\tau)$ be a finite tracial von Neumann algebra.  We define
	bounded backgrounds in $M$ recursively as follows.
	
	\begin{enumerate}[label=(\roman*)]
		\item The empty family $\varnothing$ is a bounded background, and $
		P_\varnothing=0.$
		
		\item Suppose that $\Omega_1\subset\mathcal S_0(M)$ is a bounded
		background.  We call $\Omega_2=\Omega_1\dot\cup U$
		a bounded background, and refer to \(U\) as the new layer if there exist a bounded background
		$\Theta\subset\Omega_1$, an orthonormal family
		$U\subset\mathcal S_0(M)$ satisfying $U\perp\Omega_1$,
		and von Neumann subalgebras $C\subset R\subset M$
		such that $(C,R,\Theta\cup U)$ is a completed block.
		In this case, define
		\begin{equation}\label{eq:background-projection}
			P_{\Omega_2}x
			=
			P_{\Omega_1}x
			+\bigl(E_Rx-E_C(E_Rx)\bigr)
			-P_\Theta x,
			\qquad x\in L^2(M)_{\sa}.
		\end{equation}
		
		\item Only finitely many applications of {\rm (ii)} are allowed.
	\end{enumerate}
\end{defin}

\begin{rem}
	Every completed block gives a bounded background.  Indeed, if
	$(A,Q,T)$ is a completed block, take
	$\Omega_1=\Theta=\varnothing$ and $U=T$ in
	Definition~\ref{def:background}.  Then $\Omega_2=T$, so $T$ is a
	bounded background.  Moreover,
	\[
	P_T=E_Q-E_AE_Q
	\]
	by \eqref{eq:block-projection}.
\end{rem}

\begin{rem}
	The operator defined in \eqref{eq:background-projection} is precisely the
	real $L^2$-orthogonal projection onto
	$\overline{\spanop_{\RR}}\Omega_2$.  Indeed, since
	$(C,R,\Theta\cup U)$ is a completed block,
	\eqref{eq:block-projection} gives
	\[
	E_R-E_CE_R
	=
	P_{\Theta\cup U}
	=
	P_\Theta+P_U.
	\]
	Hence \eqref{eq:background-projection} reduces to $
	P_{\Omega_2}=P_{\Omega_1}+P_U.$
	Since $U\perp\Omega_1$, the right-hand side is the orthogonal projection
	onto $
	\overline{\spanop_{\RR}}(\Omega_1\dot\cup U)
	=
	\overline{\spanop_{\RR}}\Omega_2.$
	Furthermore, induction on the finite construction shows that
	\[
	P_\Omega(M_{\sa})\subset M_{\sa}
	\]
	for every bounded background $\Omega$.  The term ``bounded'' refers to
	this property, rather than to any finiteness assumption on $\Omega$.
\end{rem}

\subsection{Closure of bounded backgrounds}
\label{subsec:background-closure}

A subfamily of a bounded background need not itself be a bounded
background, since the completed-block structure used to represent its
orthogonal projection may be lost after passing to a subset.  We therefore
introduce a closure operation which restores the necessary completed-block
structure while preserving cardinality.

This closure construction is conceptually related to the
elementary-submodel and projectional-skeleton methods used for separable
reduction in Banach spaces; see \cite{Kub09,CK14}. The proof below is
direct: the closure retains conditional-expectation images and full
Fourier supports so that the selected symmetries still determine
completed subblocks. No elementary-submodel theorem is used.

\begin{prop}\label{prop:OA-closure}
	Let $(M,\tau)$ be a finite tracial von Neumann algebra.  For every
	bounded background $\Omega\subset\mathcal S_0(M)$, there is a closure
	operation
	\[
	\cl_\Omega:S\longmapsto\cl_\Omega(S),
	\qquad S\subset\Omega,
	\]
	such that
	\begin{enumerate}
\item[(i)] $\cl_\Omega(\varnothing)=\varnothing$,
$S\subset\cl_\Omega(S)$, and
$\cl_\Omega(\cl_\Omega(S))=\cl_\Omega(S)$;
		
		\item[(ii)] $|\cl_\Omega(S)|\leq\max\{|S|,\aleph_0\}$;
		
		\item[(iii)] $\cl_\Omega$ is monotone and, for every directed increasing
		family $(S_i)_{i\in I}$,
		\[
		\cl_\Omega\left(\bigcup_iS_i\right)
		=
		\bigcup_i\cl_\Omega(S_i);
		\]
		
		\item[(iv)] every $\cl_\Omega$-closed subfamily
		$S=\cl_\Omega(S)$ is itself a bounded background.  In particular,
		\[
		P_S(M_{\sa})\subset M_{\sa}.
		\]
	\end{enumerate}
	
Moreover, suppose that $\Omega_2=\Omega_1\dot\cup U$ is obtained
from $\Omega_1$ by one application of
Definition~\ref{def:background}(ii).
Once a closure $\cl_{\Omega_1}$ satisfying {\rm (i)--(iv)} has
been fixed, a closure $\cl_{\Omega_2}$ satisfying {\rm (i)--(iv)}
can be chosen so that every $\cl_{\Omega_2}$-closed subfamily
$S\subset\Omega_2$ satisfies
\[
S\cap\Omega_1
=
\cl_{\Omega_1}(S\cap\Omega_1).
\]
\end{prop}

\begin{proof}
	We first construct the closure inside a single completed block.  We then
	assemble these one-block closures along the finite recursive construction
	of a bounded background.
	
	\medskip
	\noindent
	\emph{Step 1: a single completed block.}
	Let $\mathfrak B=(A,Q,T)$ be a completed block and let $S\subset T$.
	Our goal is to enlarge $S$ to a family $S_\infty\subset T$ and to
	construct von Neumann subalgebras $A_S\subset Q_S\subset Q$ such that
	\[
	L^2(Q_S)_{\sa}
	=
	L^2(A_S)_{\sa}
	\oplus
	\overline{\spanop_{\RR}}S_\infty.
	\]
	Thus $(A_S,Q_S,S_\infty)$ will again be a completed block.
	
	Put $D_0=\{\one\}$ and $S_0=S$.  Suppose that
	$D_n\subset A_{\sa}$ and $S_n\subset T$ have been defined for some
	$n\geq0$.  Let $\mathcal W_n$ be the set of real and imaginary parts
	of all $*$-polynomials, with coefficients in $\mathbb Q+i\mathbb Q$,
	in the elements of $D_n\cup S_n$.
	
	For each $w\in\mathcal W_n$, the completed-block decomposition gives
	\[
	w=E_A(w)+\bigl(w-E_A(w)\bigr),
	\qquad
	w-E_A(w)\in\overline{\spanop_{\RR}}T.
	\]
	To ensure that this decomposition remains available in the smaller
	structure that we are constructing, we retain $E_A(w)$ and all the
	$T$-coordinates occurring in $w-E_A(w)$.  Accordingly, define
	\begin{align*}
		D_{n+1}
		&=
		D_n\cup\{E_A(w):w\in\mathcal W_n\},\\
		S_{n+1}
		&=
		S_n\cup
		\bigcup_{w\in\mathcal W_n}
		\supp_T\bigl(w-E_A(w)\bigr).
	\end{align*}
	By Lemma~\ref{lem:density-facts}(iv), each support appearing above is
	countable.
	We also use the following support observation: if
	$\xi\in\overline{\spanop_{\RR}}T$ and
	$\supp_T(\xi)\subset F\subset T$, then
	$\xi\in\overline{\spanop_{\RR}}F$.
	Indeed, $\xi-P_F\xi$ belongs to
	$\overline{\spanop_{\RR}}T$ and is orthogonal to every member of $T$,
	so it is zero. Thus the whole support retained at a stage recovers
	the corresponding residual. A task has only finitely many generating
	parameters, although its output support may be countable. Since
	$E_A$ and $T$ are fixed, a given task has the same output at every
	stage.
	
	Now set
	\[
	D_\infty=\bigcup_{n=0}^{\infty}D_n,
	\qquad
	\cl_{\mathfrak B}(S)
	=
	S_\infty
	:=
	\bigcup_{n=0}^{\infty}S_n,
	\]
	and define
	\[
	A_S=W^*(D_\infty),
	\qquad
	Q_S=W^*(D_\infty,S_\infty).
	\]
	We claim that
	\begin{equation}\label{eq:subblock}
		L^2(Q_S)_{\sa}
		=
		L^2(A_S)_{\sa}
		\oplus
		\overline{\spanop_{\RR}}S_\infty.
	\end{equation}
	
To prove the claim, let $w$ be a self-adjoint $*$-polynomial over
$\mathbb Q+i\mathbb Q$ in elements of $D_\infty\cup S_\infty$.
Since $w$ involves only finitely many elements, there exists
$n_0\geq0$ such that all of them belong to
$D_{n_0}\cup S_{n_0}$.  The definition of the next stage then gives
\[
E_A(w)\in D_{n_0+1}\subset A_S,
\qquad
\supp_T\bigl(w-E_A(w)\bigr)
\subset S_{n_0+1}\subset S_\infty.
\]
	Hence
	\[
	E_A(w)\in A_S,
	\qquad
	w-E_A(w)\in\overline{\spanop_{\RR}}S_\infty.
	\]
	
Let $\mathcal A_{\mathbb Q}$ be the unital $*$-algebra over
$\mathbb Q+i\mathbb Q$ generated by $D_\infty\cup S_\infty$.
Its norm closure is $C^*(D_\infty\cup S_\infty)$. By Kaplansky's
density theorem
\cite[p.~329, Theorem~5.3.5]{KR97a}, every $x=x^*\in Q_S$ is a
strong operator limit of a uniformly bounded net of self-adjoint
elements of $C^*(D_\infty\cup S_\infty)$. Approximating the members
of this net in norm by self-adjoint elements of
$\mathcal A_{\mathbb Q}$, and passing to the product directed set,
we obtain a uniformly bounded net
\[
(w_j)\subset(\mathcal A_{\mathbb Q})_{\sa}
\]
such that $w_j\to x$ strongly. In the standard representation,
$w_j\to x$ in $\|\cdot\|_2$.
	
	For every $j$, the preceding finite-polynomial argument gives
	\[
	E_A(w_j)\in A_S,
	\qquad
	w_j-E_A(w_j)\in\overline{\spanop_{\RR}}S_\infty.
	\]
	Since the $L^2$-extension of $E_A$ is contractive,
	\[
	E_A(w_j)\longrightarrow E_A(x),
	\qquad
	w_j-E_A(w_j)\longrightarrow x-E_A(x)
	\]
	in $\|\cdot\|_2$.  Therefore
	\[
	E_A(x)\in L^2(A_S)_{\sa},
	\qquad
	x-E_A(x)\in\overline{\spanop_{\RR}}S_\infty.
	\]
	Thus
	\[
	L^2(Q_S)_{\sa}
	\subset
	L^2(A_S)_{\sa}
	+
	\overline{\spanop_{\RR}}S_\infty.
	\]
	The reverse inclusion follows from the definitions of $A_S$ and $Q_S$.
	Moreover, the two summands are orthogonal because
	$A_S\subset A$ and $S_\infty\subset T$, while
	\[
	L^2(A)_{\sa}\perp
	\overline{\spanop_{\RR}}T.
	\]
	This proves \eqref{eq:subblock}.  Hence
	$(A_S,Q_S,S_\infty)$ is a completed block.
	
	We next record the closure properties of this construction.  Put
	\[
	\lambda=\max\{|S|,\aleph_0\}.
	\]
	We claim that
	\[
	|D_n|,\ |S_n|\leq\lambda
	\qquad(n\geq0).
	\]
	This is clear for $n=0$.  If it holds at stage $n$, then the set of
	$(\mathbb Q+i\mathbb Q)$-coefficient $*$-polynomials in
	$D_n\cup S_n$ has cardinality at most $\lambda$.  Each such polynomial
	contributes one element to $D_{n+1}$ and at most countably many elements
	to $S_{n+1}$.  Since $\lambda$ is infinite,
	\[
	\lambda\cdot\aleph_0=\lambda,
	\]
	and therefore the same bound holds at stage $n+1$.  Consequently,
	\[
	|\cl_{\mathfrak B}(S)|
	=
	\left|\bigcup_{n=0}^{\infty}S_n\right|
	\leq\lambda.
	\]
	This proves the required cardinal estimate for the one-block closure.
	
	The construction is also extensive and monotone.  To verify idempotence,
	start the same construction again from $S_\infty$ and denote the
	resulting sequences by $(D'_n)$ and $(S'_n)$.  We show inductively that
	\[
	D'_n\subset D_\infty,
	\qquad
	S'_n\subset S_\infty
	\qquad(n\geq0).
	\]
	For $n=0$ this follows from
	$D'_0=\{\one\}\subset D_\infty$ and $S'_0=S_\infty$.
	Suppose it holds at stage $n$.  Every polynomial used to form
	$D'_{n+1}$ and $S'_{n+1}$ involves only finitely many elements of
	$D_\infty\cup S_\infty$.  Hence all of these elements occur in
	$D_m\cup S_m$ for some $m$, and the original construction already
	adjoined its $E_A$-image and its full $T$-support at stage $m+1$.
	Thus
	\[
	D'_{n+1}\subset D_\infty,
	\qquad
	S'_{n+1}\subset S_\infty.
	\]
	It follows that
	\[
	\cl_{\mathfrak B}(S_\infty)
	=
	\bigcup_{n=0}^{\infty}S'_n
	\subset S_\infty.
	\]
	The reverse inclusion follows from extensivity.  Therefore
	\[
	\cl_{\mathfrak B}
	\bigl(\cl_{\mathfrak B}(S)\bigr)
	=
	\cl_{\mathfrak B}(S).
	\]
	
	Finally, let $(S_i)_{i\in I}$ be a nonempty directed increasing family of
	subsets of $T$.  We claim that for every $n\geq0$,
	\[
	D_n\!\left(\bigcup_iS_i\right)
	=
	\bigcup_iD_n(S_i),
	\qquad
	S_n\!\left(\bigcup_iS_i\right)
	=
	\bigcup_iS_n(S_i).
	\]
	For $n=0$ this follows directly from the definitions.  Suppose the
	equalities hold at stage $n$.  Any polynomial occurring in the
	construction from $\bigcup_iS_i$ uses only finitely many elements of
	\[
	\bigcup_iD_n(S_i)
	\quad\text{and}\quad
	\bigcup_iS_n(S_i).
	\]
	By directedness, all these finitely many elements occur at one common
	index $i$.  Thus every polynomial used at stage $n+1$ for the union
	already occurs at stage $n+1$ for some $S_i$, and the reverse inclusion
	follows from monotonicity.  This proves the claim.
	
	Recalling that
	\[
	\cl_{\mathfrak B}(S)
	=
	\bigcup_{n=0}^{\infty}S_n(S),
	\]
	we obtain
	\[
	\begin{aligned}
		\cl_{\mathfrak B}\left(\bigcup_iS_i\right)
		&=
		\bigcup_{n=0}^{\infty}
		S_n\left(\bigcup_iS_i\right)\\
		&=
		\bigcup_{n=0}^{\infty}\bigcup_iS_n(S_i)\\
		&=
		\bigcup_i\bigcup_{n=0}^{\infty}S_n(S_i)\\
		&=
		\bigcup_i\cl_{\mathfrak B}(S_i).
	\end{aligned}
	\]
	Thus the one-block closure is continuous along directed increasing
	families.
	
	\medskip
	\noindent
	\emph{Step 2: a general bounded background.}
We use induction on the height of a finite derivation tree witnessing
that $\Omega$ is a bounded background. We take this tree to include,
at each application of Definition~\ref{def:background}(ii), the chosen
derivations witnessing that both the old background and the auxiliary
background are bounded backgrounds.  Thus the derivation trees associated with
$\Omega^-$ and $\Theta$ have strictly smaller height.  If
	$\Omega=\varnothing$, there is nothing to prove.
	
	For the inductive step, suppose that the last application of
	Definition~\ref{def:background}(ii) has the form
	\[
	\Omega=\Omega^-\dot\cup U,
	\]
	where $\Omega^-$ and $\Theta\subset\Omega^-$ are bounded backgrounds,
	$U\perp\Omega^-$, and
	\[
	\mathfrak B=(C,R,\Theta\cup U)
	\]
	is a completed block.  By the induction hypothesis, closure operations
	are already available on $\Omega^-$ and on $\Theta$, while Step~1 gives
	the one-block closure on $\Theta\cup U$.
	
	Let $S\subset\Omega$.  In order that an enlargement of $S$ retain the
	bounded-background structure, three conditions have to be enforced:
	$S\cap\Omega^-$ must be closed inside the old background,
	$S\cap\Theta$ must be closed inside $\Theta$, and
	$(S\cap\Theta)\cup(S\cap U)$ must determine a completed subblock of
	$\mathfrak B$.  We therefore define
	\begin{align*}
		F(S)=S
		&\cup\cl_{\Omega^-}(S\cap\Omega^-)
		\cup\cl_\Theta(S\cap\Theta)\\
		&\cup\cl_{\mathfrak B}
		\bigl((S\cap\Theta)\cup(S\cap U)\bigr).
	\end{align*}
	
	Starting from $S$, iterate this operation by
	\[
	F^0(S)=S,
	\qquad
	F^{n+1}(S)=F(F^n(S)),
	\qquad n\geq0,
	\]
	and define
	\begin{equation}\label{eq:background-closure}
		\cl_\Omega(S)
		=
		\bigcup_{n=0}^{\infty}F^n(S).
	\end{equation}
	The iteration is needed because closing one of the three pieces may add
	new elements which force further closure in another piece.
	
	We now verify the stated properties.  The map $F$ is extensive and
	monotone, since each of the three closure operations appearing in its
	definition has these properties.  Put $\lambda=\max\{|S|,\aleph_0\}.$
	If $|F^n(S)|\leq\lambda$, then each of the three closures used to form
	$F^{n+1}(S)$ has cardinality at most $\lambda$, and hence
	\[
	|F^{n+1}(S)|\leq\lambda.
	\]
	By induction this holds for every $n\geq0$, and therefore
	\[
	|\cl_\Omega(S)|
	=
	\left|\bigcup_{n=0}^{\infty}F^n(S)\right|
	\leq\lambda.
	\]
	This proves (ii), while extensivity and monotonicity follow immediately
	from the construction.
	
	Next we show that $\cl_\Omega(S)$ is a fixed point of $F$.  Since
	$(F^n(S))_{n\geq0}$ is increasing and $F$ is continuous along directed
	increasing families,
	\[
	\begin{aligned}
		F(\cl_\Omega(S))
		&=
		F\left(\bigcup_{n=0}^{\infty}F^n(S)\right)\\
		&=
		\bigcup_{n=0}^{\infty}F(F^n(S))\\
		&=
		\bigcup_{n=0}^{\infty}F^{n+1}(S)\\
		&=
		\bigcup_{n=0}^{\infty}F^n(S)\\
		&=
		\cl_\Omega(S).
	\end{aligned}
	\]
	In the fourth equality we used the extensivity of $F$, which gives
	$F^0(S)\subset F^1(S)$.  Since $\cl_\Omega(S)$ is already a fixed point
	of $F$, iterating $F$ once again produces no new elements.  Thus
	\[
	\cl_\Omega(\cl_\Omega(S))
	=
	\cl_\Omega(S),
	\]
	and (i) follows.
	
	We next prove directed continuity.  Let $(S_i)_{i\in I}$ be a directed
	increasing family.  Since $F$ is continuous along directed increasing
	families, induction on $n$ gives
	\[
	F^n\left(\bigcup_iS_i\right)
	=
	\bigcup_iF^n(S_i)
	\qquad(n\geq0).
	\]
	Using \eqref{eq:background-closure}, we therefore obtain
	\[
	\begin{aligned}
		\cl_\Omega\left(\bigcup_iS_i\right)
		&=
		\bigcup_{n=0}^{\infty}
		F^n\left(\bigcup_iS_i\right)\\
		&=
		\bigcup_{n=0}^{\infty}\bigcup_iF^n(S_i)\\
		&=
		\bigcup_i\bigcup_{n=0}^{\infty}F^n(S_i)\\
		&=
		\bigcup_i\cl_\Omega(S_i).
	\end{aligned}
	\]
	This proves (iii).
	
	It remains to prove (iv).  Suppose that $
	S=\cl_\Omega(S).$
	Since $F(S)$ is one of the sets occurring in
	$\bigcup_{n=0}^{\infty}F^n(S)$, we have
	\[
	F(S)\subset\cl_\Omega(S)=S.
	\]
	On the other hand, $F$ is extensive, so $S\subset F(S)$.  Hence
	\[
	F(S)=S.
	\]
	
	From the definition of $F$ we then obtain
	\[
	\cl_{\Omega^-}(S\cap\Omega^-)\subset S\cap\Omega^-.
	\]
	The reverse inclusion follows from extensivity of
	$\cl_{\Omega^-}$, and therefore
	\[
	S\cap\Omega^-
	=
	\cl_{\Omega^-}(S\cap\Omega^-).
	\]
	Similarly,
	\[
	S\cap\Theta
	=
	\cl_\Theta(S\cap\Theta).
	\]
	By the induction hypothesis, both $S\cap\Omega^-$ and
	$S\cap\Theta$ are bounded backgrounds.
	
	The same argument applied to the one-block closure gives
	\[
	(S\cap\Theta)\cup(S\cap U)
	=
	\cl_{\mathfrak B}
	\bigl((S\cap\Theta)\cup(S\cap U)\bigr).
	\]
	By Step~1 and \eqref{eq:subblock}, there are von Neumann subalgebras
	$C_S\subset R_S\subset M$ such that
	\[
	\bigl(C_S,R_S,(S\cap\Theta)\cup(S\cap U)\bigr)
	\]
	is a completed block.
	
	Finally,
	\[
	S\cap U\perp S\cap\Omega^-,
	\]
	because $U\perp\Omega^-$.  We may therefore apply
	Definition~\ref{def:background}(ii) to the bounded background
	$S\cap\Omega^-$, with bounded subbackground $S\cap\Theta$ and new layer
	$S\cap U$.  It follows that
	\[
	S
	=
	(S\cap\Omega^-)\dot\cup(S\cap U)
	\]
	is a bounded background.  Consequently,
	\[
	P_S(M_{\sa})\subset M_{\sa}.
	\]
	This proves (iv).
	
The construction also fixes the empty family. In Step~1, if
$S_n=\varnothing$, then every task belongs to $A$, since
$D_n\subset A$. Its $E_A$-image is therefore the task itself,
and its residual has empty support. Starting from $S=\varnothing$,
induction gives $S_n=\varnothing$ at every stage. In Step~2, the
constituent closures fix the empty family by induction, so
$F(\varnothing)=\varnothing$ and hence
$\cl_\Omega(\varnothing)=\varnothing$.
If directed families with empty index sets are allowed, their
union identity in {\rm (iii)} follows from this equality.

Finally, suppose that $\Omega=\Omega^-\dot\cup U$ is obtained
by adding one new layer, and prescribe any closure
$\cl_{\Omega^-}$ satisfying {\rm (i)--(iv)}.
Choose a closure on $\Theta$ by the first part of the proposition
and use the prescribed $\cl_{\Omega^-}$ in Step~2.
The verification remains valid, since it uses only
{\rm (i)--(iv)} for the closures on $\Omega^-$ and $\Theta$.
Every $\cl_\Omega$-closed family $S$ is therefore a fixed point
of $F$, and consequently
\[
S\cap\Omega^-
=
\cl_{\Omega^-}(S\cap\Omega^-).
\]
This proves compatibility with the prescribed closure on the
old background.

\end{proof}

	\section{Relative basis extension}\label{sec:extension}

We first prove the countable extension step.  The basic idea is to treat
successively all self-adjoint polynomials generated during the construction.
For such a polynomial $y$, we use the orthogonal decomposition
\[
y
=
E_N(y)+P_\Omega(y)+b(y),
\qquad
b(y):=y-E_N(y)-P_\Omega(y).
\]
The first term will be absorbed into a countably generated subalgebra of
$N$, the second term will be supported on a countable closed subbackground
of $\Omega$, and the residual $b(y)$ will be approximated by newly chosen
symmetries.  Repeating this for every polynomial and with arbitrarily small
error produces the required completed block.

\begin{lem}\label{lem:countable-layer}Let $(M,\tau)$ be a finite tracial von Neumann algebra, set
	$\kappa=\dens L^2(M,\tau)$, and assume that
	$\kappa>\aleph_0$ and that $(M,\tau)$ satisfies \eqref{eq:CD}.
	Let $N\subset M$  be a von Neumann subalgebra  satisfying $\dens L^2(N)<\kappa$.  Let $\Omega$ be a
	bounded background with
	\[
	|\Omega|<\kappa,
	\qquad
	E_N(\Omega)=0,
	\]
	and fix the closure operation $\cl_\Omega$ from
	Proposition~\ref{prop:OA-closure}.  If
	$S\subset\Omega$, $X\subset M_{\sa}$, and $Y\subset N_{\sa}$ are
	countable, then there exist von Neumann algebras
	$C\subset N$ and $C\subset R\subset M$, a countable $\cl_\Omega$-closed family
	$\Omega_\infty\subset\Omega$, and a countable
	orthonormal family $U\subset\mathcal S_0(M)$ such that
	\[
S\subset\Omega_\infty,
	\qquad
	X\subset R,
	\qquad
	Y\subset C,
	\]
	\[
	E_N(U)=0,
	\qquad
	U\perp\Omega,
	\]
	and
	\begin{equation}\label{eq:countable-block}
	L^2(R)_{\sa}
	=
	L^2(C)_{\sa}
	\oplus
	\overline{\spanop_{\RR}}(\Omega_\infty\cup U).
	\end{equation}
	In particular,
	\[
	\dens L^2(R)\leq\aleph_0.
	\]
\end{lem}

\begin{proof}
	We construct increasing families
	\[
	\mathcal C_n\subset N_{\sa},
	\qquad
	\Omega_n\subset\Omega,
	\qquad
	U_n\subset\mathcal S_0(M),
	\qquad n=0,1,2,\ldots,
	\]
	where $\mathcal C_n$ and $\Omega_n$ are countable,
	$\Omega_n$ is $\cl_\Omega$-closed, and $U_n$ is finite and
	orthonormal. At every stage we maintain
	\[
	E_N(U_n)=0,
	\qquad
	U_n\perp\Omega.
	\]
	
	Start with
	\[
	\mathcal C_0=Y\cup\{\one\},
	\qquad
	\Omega_0=\cl_\Omega(S),
	\qquad
	U_0=\varnothing.
	\]
	By Proposition~\ref{prop:OA-closure},
	$\Omega_0$ is countable and $\cl_\Omega$-closed.
	
	For each $n\geq0$, let
	\[
	\mathcal G_n
	=
	\mathcal C_n\cup X\cup\Omega_n\cup U_n
	\]
	be the set of parameters available at stage $n$. Each
	$\mathcal G_n$ is countable. Let $\mathcal T_n$ be the set of all
	pairs $(y,k)$, where $k\geq1$ and $y$ is the real or imaginary part
	of a $(\mathbb Q+i\mathbb Q)$-coefficient $*$-polynomial in finitely
	many elements of $\mathcal G_n$. Thus $\mathcal T_n$ is countable.
	
	We now specify a schedule that guarantees that every task is
	eventually processed. For each $n$, choose an enumeration
	\[
	\mathcal T_n=\{t_{n,j}:j=0,1,2,\ldots\},
	\]
	allowing repetitions if necessary. Fix a bijection
	\[
	\pi:\mathbb N_0\times\mathbb N_0\longrightarrow\mathbb N_0,
	\qquad
	\mathbb N_0=\{0,1,2,\ldots\},
	\]
	such that $\pi(m,j)\geq m$. At stage $n$, write
	\[
	\pi^{-1}(n)=(m,j)
	\]
	and process the task
	\[
	t_{m,j}=(y,k).
	\]
	Since $m\leq n$, the parameter set $\mathcal G_m$ has already been
	constructed when this task is processed. Moreover, every task
	belonging to any $\mathcal T_m$ is processed at a finite stage.
	The enumeration of $\mathcal T_m$ is chosen when $\mathcal G_m$ is
	formed and is retained thereafter. Such an enumeration always exists:
	$(0,1)\in\mathcal T_m$, even if the initial data are empty.
	If $t_{m,j}=(y,k)$, its approximation is available after stage
	$\pi(m,j)$, in the family $U_{\pi(m,j)+1}$.
	Thus every fixed self-adjoint polynomial from any earlier stage is treated at
	every prescribed accuracy.
	
	For the chosen task $(y,k)$, define
	\begin{equation}\label{eq:bounded-residual}
		b(y)
		=
		y-E_N(y)-P_\Omega(y).
	\end{equation}
	Since $\Omega$ is a bounded background,
	\[
	P_\Omega(y)\in M_{\sa},
	\]
	and hence $b(y)\in M_{\sa}$. The operators $E_N$ and $P_\Omega$
	remain fixed throughout the construction, so a fixed $y$ gives the
	same bounded residual for every task $(y,k)$.
	The orthogonality constraints below use the entire family $\Omega$,
	independently of the current $\Omega_n$.
	Moreover, $E_N(\Omega)=0$ implies
	\[
	L^2(N)_{\sa}
	\perp
	\overline{\spanop_{\RR}}\Omega.
	\]
	Therefore
	\[
	E_N(b(y))=0,
	\qquad
	b(y)\perp\Omega.
	\]
	
	The three terms in
	\[
	y=E_N(y)+P_\Omega(y)+b(y)
	\]
	are treated separately. First, put
	\[
	\mathcal C_{n+1}
	=
	\mathcal C_n\cup\{E_N(y)\}.
	\]
	Next, let
	\[
	\Sigma(y)
	=
	\supp_\Omega(P_\Omega(y)).
	\]
	By Lemma~\ref{lem:density-facts}(iv), $\Sigma(y)$ is countable.
	Define
	\[
	\Omega_{n+1}
	=
	\cl_\Omega\bigl(\Omega_n\cup\Sigma(y)\bigr).
	\]
	Then $\Omega_{n+1}$ is countable and $\cl_\Omega$-closed, and it
	contains all the $\Omega$-coordinates needed to represent
	$P_\Omega(y)$.
	
	Finally, let
	\[
	V_n=\spanop_{\RR}U_n.
	\]
	Since only finitely many symmetries have been chosen before stage
	$n$, the space $V_n$ is finite-dimensional. Apply
	Lemma~\ref{lem:finite-approximation} with
	\[
	\Gamma=\Omega,
	\qquad
	b=b(y),
	\qquad
	V=V_n,
	\qquad
	\varepsilon=2^{-k}.
	\]
	The hypotheses hold because
	\[
	|\Omega|<\kappa,
	\qquad
	E_N(\Omega)=0,
	\qquad
	E_N(b(y))=0,
	\qquad
	b(y)\perp\Omega.
	\]
	Thus there is a finite orthonormal family
	$W_n\subset\mathcal S_0(M)$, orthogonal to $U_n$, such that
	\[
	E_N(W_n)=0,
	\qquad
	W_n\perp\Omega,
	\]
	and
	\[
	\dist\left(
	b(y),
	V_n+\spanop_{\RR}W_n
	\right)
	<
	2^{-k}.
	\]
	Set
	\[
	U_{n+1}=U_n\dot\cup W_n.
	\]
	This completes stage $n$. The newly added elements of
	$\mathcal C_{n+1}$, $\Omega_{n+1}$, and $U_{n+1}$ are included in
	$\mathcal G_{n+1}$ and hence participate in the polynomial tasks at
	all later stages. The update at a single stage is summarized in
	Figure~\ref{fig:countable-layer}.
	
	\begin{figure}[htbp]
		\centering
		\resizebox{\textwidth}{!}{%
			\begin{tikzpicture}[
				node distance=8mm,
				every node/.style={align=center},
				box/.style={
					draw,
					rectangle,
					rounded corners=1pt,
					inner xsep=8pt,
					inner ysep=5pt
				},
				arr/.style={-{Latex[length=2mm]},thick}
				]
				\node[box] (gen)
				{current countable parameter set $\mathcal G_n$};
				
				\node[box, below=of gen] (poly)
				{choose a task $(y,k)$};
				
				\node[box, below=of poly] (dec)
				{$y=E_N(y)+P_\Omega(y)+b(y)$};
				
				\node[box, below left=12mm and 20mm of dec] (c)
				{adjoin $E_N(y)$\\to $\mathcal C_{n+1}$};
				
				\node[box, below=12mm of dec] (o)
				{adjoin the support of $P_\Omega(y)$\\
					to $\Omega_{n+1}$};
				
				\node[box, below right=12mm and 20mm of dec] (u)
				{approximate $b(y)$\\by new symmetries};
				
				\node[box, below=20mm of o] (next)
				{enlarged parameter set $\mathcal G_{n+1}$};
				
				\draw[arr] (gen) -- (poly);
				\draw[arr] (poly) -- (dec);
				\draw[arr] (dec) -- (c);
				\draw[arr] (dec) -- (o);
				\draw[arr] (dec) -- (u);
				\draw[arr] (c) -- (next);
				\draw[arr] (o) -- (next);
				\draw[arr] (u) -- (next);
			\end{tikzpicture}%
		}
		\caption{One stage of the countable extension construction.}
		\label{fig:countable-layer}
	\end{figure}
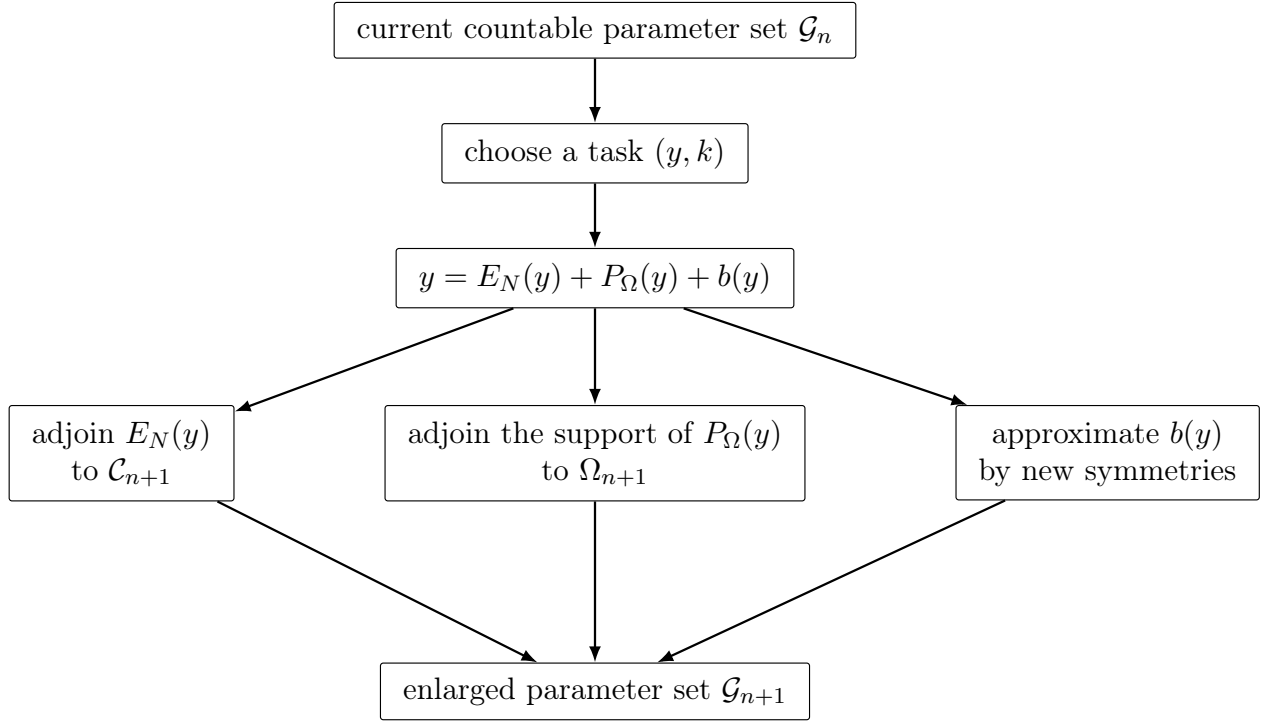
	
	After completing all stages, put
	\[
	\mathcal C=\bigcup_{n=0}^{\infty}\mathcal C_n,
	\qquad
	\Omega_\infty=\bigcup_{n=0}^{\infty}\Omega_n,
	\qquad
	U=\bigcup_{n=0}^{\infty}U_n,
	\]
	and define
	\[
	C=W^*(\mathcal C),
	\qquad
	R=W^*(\mathcal C\cup X\cup\Omega_\infty\cup U).
	\]
	The generating set of $R$ is countable. Hence
	Lemma~\ref{lem:density-facts}(i) gives
	\[
	\dens L^2(R)\leq\aleph_0.
	\]
	Also,
	\[
	Y\subset\mathcal C\subset C,
	\qquad
	X\subset R,
	\qquad
	S\subset\Omega_\infty.
	\]
	
	We next verify that $\Omega_\infty$ is $\cl_\Omega$-closed. Since
	every $\Omega_n$ is $\cl_\Omega$-closed and the sequence
	$(\Omega_n)$ is increasing, Proposition~\ref{prop:OA-closure}(iii)
	gives
	\[
	\begin{aligned}
		\cl_\Omega(\Omega_\infty)
		&=
		\cl_\Omega\left(
		\bigcup_{n=0}^{\infty}\Omega_n
		\right)\\
		&=
		\bigcup_{n=0}^{\infty}\cl_\Omega(\Omega_n)\\
		&=
		\bigcup_{n=0}^{\infty}\Omega_n\\
		&=
		\Omega_\infty.
	\end{aligned}
	\]
	Thus $\Omega_\infty$ is a countable closed subbackground.
	
	The construction also gives
	\[
	E_N(U)=0,
	\qquad
	U\perp\Omega,
	\]
	and $U$ is orthonormal, because at every stage the newly chosen
	symmetries are orthogonal to all previously chosen ones.
	
	It remains to prove the decomposition
	\eqref{eq:countable-block}. Let $y$ be the real or imaginary part
	of a $(\mathbb Q+i\mathbb Q)$-coefficient $*$-polynomial in
	finitely many elements of
	\[
	\mathcal C\cup X\cup\Omega_\infty\cup U.
	\]
Only finitely many parameters occur in $y$, so there exists
$m_0\geq0$ such that all of them already belong to
$\mathcal G_{m_0}$. Consequently, for every $k\geq1$, the task
$(y,k)$ belongs to $\mathcal T_{m_0}$ and is therefore processed at
some finite stage.
	
	When $(y,k)$ is processed, the construction gives
	\[
	\dist\left(
	b(y),
	\spanop_{\RR}U
	\right)
	<
	2^{-k},
	\]
	since the finite-dimensional approximation obtained at that stage
	is contained in the final span of $U$. As this holds for every
	$k\geq1$,
	\[
	\dist\left(
	b(y),
	\overline{\spanop_{\RR}}U
	\right)=0,
	\]
	and therefore
	\[
	b(y)\in\overline{\spanop_{\RR}}U.
	\]
	
	At the same time, processing any one of the tasks $(y,k)$ places
	$E_N(y)$ in $\mathcal C$, so
	\[
	E_N(y)\in C.
	\]
	It also places the entire support
	$\supp_\Omega(P_\Omega(y))$ inside some
	$\Omega_{n+1}\subset\Omega_\infty$. Hence
	\[
	P_\Omega(y)
	\in
	\overline{\spanop_{\RR}}\Omega_\infty.
	\]
	Using \eqref{eq:bounded-residual}, we obtain
	\[
	\begin{aligned}
		y
		&=
		E_N(y)+P_\Omega(y)+b(y)\\
		&\in
		L^2(C)_{\sa}
		+
		\overline{\spanop_{\RR}}(\Omega_\infty\cup U).
	\end{aligned}
	\]
	
	The real and imaginary parts of
	$(\mathbb Q+i\mathbb Q)$-coefficient $*$-polynomials in
	\[
	\mathcal C\cup X\cup\Omega_\infty\cup U
	\]
	are $L^2$-dense in $L^2(R)_{\sa}$. Indeed, the underlying unital
	$*$-algebra is norm dense in the corresponding $C^*$-algebra, and
	Kaplansky's density theorem then gives $L^2$-density in $R$, exactly
	as in Lemma~\ref{lem:density-facts}(i).
	For the self-adjoint assertion, involution is an isometry of tracial
	$L^2$, and for $a\in R$ and $\xi\in L^2(R)_{\sa}$,
	\[
	\|\operatorname{Re}a-\xi\|_2\leq\|a-\xi\|_2.
	\]
	Taking real parts of the rational polynomial approximants therefore
	gives density in the full self-adjoint $L^2$-space. Hence
	\[
	L^2(R)_{\sa}
	\subset
	L^2(C)_{\sa}
	+
	\overline{\spanop_{\RR}}(\Omega_\infty\cup U).
	\]
	
	This sum is orthogonal. Indeed,
	\[
	C\subset N,
	\qquad
	E_N(\Omega_\infty)=0,
	\qquad
	E_N(U)=0,
	\]
	so
	\[
	L^2(C)_{\sa}
	\perp
	\overline{\spanop_{\RR}}(\Omega_\infty\cup U).
	\]
	Moreover,
	\[
	U\perp\Omega_\infty
	\]
	because $\Omega_\infty\subset\Omega$. Thus the sum is closed.
	
	Conversely,
	\[
	C\subset R,
	\qquad
	\Omega_\infty\cup U\subset R,
	\]
	by the definition of $R$. Therefore
	\[
	L^2(C)_{\sa}
	\oplus
	\overline{\spanop_{\RR}}(\Omega_\infty\cup U)
	\subset
	L^2(R)_{\sa}.
	\]
	Combining the two inclusions gives
	\[
	L^2(R)_{\sa}
	=
	L^2(C)_{\sa}
	\oplus
	\overline{\spanop_{\RR}}(\Omega_\infty\cup U),
	\]
	which is \eqref{eq:countable-block}.
\end{proof}

The countable lemma extends to any size below the ambient density.
We now give the main density-induction argument. We use that if $\lambda>\aleph_0$ and
	$\alpha<\lambda$, then $\max\{\aleph_0,|\alpha|\}<\lambda$; no regularity of
	$\lambda$ is assumed.

\begin{prop}\label{prop:small-layer}Let $(M,\tau)$ be a finite tracial von Neumann algebra, set
	$\kappa=\dens L^2(M,\tau)$, and assume that
	$\kappa>\aleph_0$ and that $(M,\tau)$ satisfies \eqref{eq:CD}.
	Let $\aleph_0\leq\lambda<\kappa$.  Suppose that \(N\subset M\) is a von Neumann subalgebra with
	$\dens L^2(N)<\kappa$, and let $\Omega$ be a bounded background with
	\[
	|\Omega|<\kappa,
	\qquad
	E_N(\Omega)=0.
	\]
	Fix the closure operation $\cl_\Omega$ from
	Proposition~\ref{prop:OA-closure}.  If
	\[
	S\subset\Omega,
	\qquad
	X\subset M_{\sa},
	\qquad
	Y\subset N_{\sa},
	\qquad
	|S|,|X|,|Y|\leq\lambda,
	\]
	then there exist von Neumann algebras $C\subset N$ and
	$C\subset R\subset M$, a $\cl_\Omega$-closed family
	$\Omega_0\subset\Omega$, and an orthonormal family
	$U\subset\mathcal S_0(M)$ such that
	\[
	S\subset\Omega_0,
	\qquad
	X\subset R,
	\qquad
	Y\subset C,
	\]
	\[
	E_N(U)=0,
	\qquad
	U\perp\Omega,
	\]
	\[
	|\Omega_0|,\ |U|,\ \dens L^2(R)\leq\lambda,
	\]
	and
	\begin{equation}\label{eq:small-block}
		L^2(R)_{\sa}
		=
		L^2(C)_{\sa}
		\oplus
		\overline{\spanop_{\RR}}(\Omega_0\cup U).
	\end{equation}
\end{prop}

\begin{proof}
Fix $(M,\tau)$ and $N$ as above. We prove the assertion
simultaneously for all admissible bounded backgrounds and all
admissible choices of $S$, $X$, and $Y$, by induction on the infinite
cardinal $\lambda$. The case $\lambda=\aleph_0$ is
Lemma~\ref{lem:countable-layer}. Suppose now that
$\lambda>\aleph_0$ and that the assertion, with arbitrary admissible
background and data, is known for every infinite cardinal
$\mu<\lambda$.

	The proof is a transfinite bookkeeping argument.  We absorb the data in
	$S$, $X$, and $Y$ one piece at a time, while maintaining a completed
	block of size at most
	\[
	\lambda_\alpha:=\max\{\aleph_0,|\alpha|\}
	\]
	at stage $\alpha<\lambda$.
	
	To organize the input, enumerate the tagged disjoint union
	\[
	(\{0\}\times S)\dot\cup
	(\{1\}\times X)\dot\cup
	(\{2\}\times Y)\dot\cup
	(\{3\}\times\{0\})
	\]
	as
	\[
	\bigl((t_\alpha,z_\alpha)\bigr)_{\alpha<\lambda},
	\]
	allowing repetitions if the tagged union has cardinality $<\lambda$.
	The last tagged datum is only a dummy symbol; it has no effect, but it
	allows us to run the recursion through all ordinals $\alpha<\lambda$.
	
	We shall construct, by transfinite recursion on $\alpha\leq\lambda$,
	increasing families
	\[
	C_\alpha\subset N,
	\qquad
	C_\alpha\subset R_\alpha\subset M,
	\qquad
	\Theta_\alpha\subset\Omega,
	\qquad
	U_\alpha\subset\mathcal S_0(M),
	\]
	such that:
	\begin{enumerate}
		\item[(a)] $\Theta_\alpha$ is $\cl_\Omega$-closed;
\item[(b)] $U_\alpha$ is an orthonormal family,
$E_N(U_\alpha)=0$, and $U_\alpha\perp\Omega$;
		\item[(c)] $\max \left\lbrace |\Theta_\alpha|,\ |U_\alpha|,\ \dens L^2(R_\alpha)\right\rbrace 
		\leq \lambda_\alpha$;
	\item[(d)] $(C_\alpha,R_\alpha,\Theta_\alpha\cup U_\alpha)$ is a
	completed block; equivalently,
	\begin{equation}\label{eq:induction-blocks-new}
		L^2(R_\alpha)_{\sa}
		=
		L^2(C_\alpha)_{\sa}
		\oplus
		\overline{\spanop_{\RR}}(\Theta_\alpha\cup U_\alpha);\end{equation}
		\item[(e)] every tagged datum indexed by $\beta<\alpha$ has already
		been absorbed, in the sense that
		\[
		z_\beta\in\Theta_\alpha \text{ if } t_\beta=0,
		\qquad
		z_\beta\in R_\alpha \text{ if } t_\beta\in\{1,3\},
		\qquad
		z_\beta\in C_\alpha \text{ if } t_\beta=2.
		\]
	\end{enumerate}
	
	The successor step of the induction is summarized in
	Figure~\ref{fig:small-layer}; its verification is given below.
	
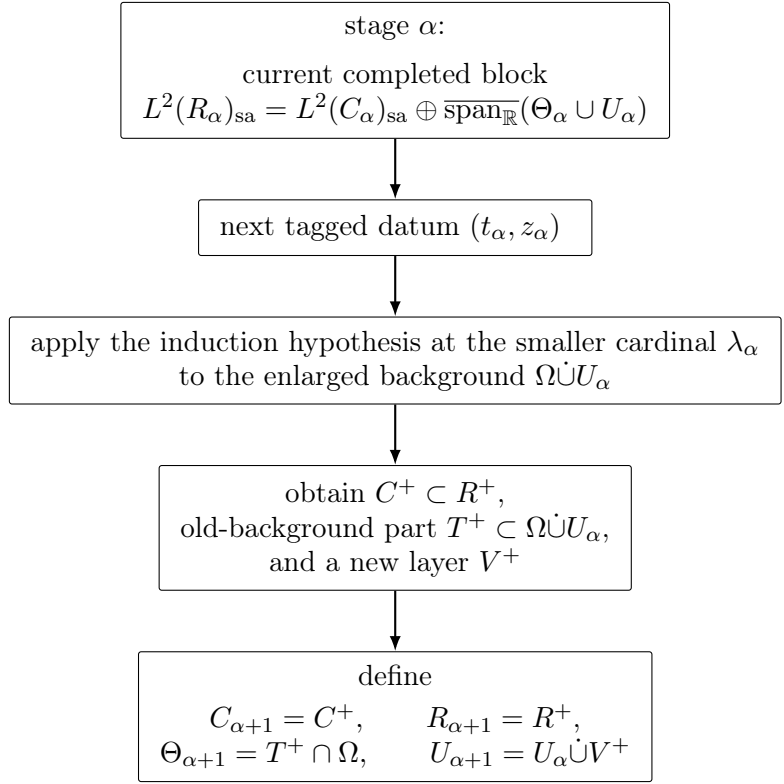
\begin{figure}[htbp]
	\centering
	\begin{tikzpicture}[
		node distance=8mm,
		box/.style={
			draw,
			rectangle,
			rounded corners=1pt,
			align=center,
			inner xsep=8pt,
			inner ysep=5pt
		},
		arr/.style={
			-{Latex[length=2mm]},
			thick
		}
		]
		
		\node[box] (init) {
			stage $\alpha$:\\[2mm]
			current completed block\\
			$L^2(R_\alpha)_{\sa}
			=
			L^2(C_\alpha)_{\sa}
			\oplus
			\overline{\spanop_{\RR}}(\Theta_\alpha\cup U_\alpha)$
		};
		
		\node[box, below=of init] (datum) {
			next tagged datum $(t_\alpha,z_\alpha)$
		};
		
		\node[box, below=of datum] (small) {
			apply the induction hypothesis at the smaller cardinal
			$\lambda_\alpha$\\
			to the enlarged background $\Omega\dot\cup U_\alpha$
		};
		
		\node[box, below=of small] (plus) {
			obtain $C^+\subset R^+$,\\
			old-background part
			$T^+\subset\Omega\dot\cup U_\alpha$,\\
			and a new layer $V^+$
		};
		
		\node[box, below=of plus] (next) {
			define\\[1mm]
			$C_{\alpha+1}=C^+,\qquad
			R_{\alpha+1}=R^+$,\\
			$\Theta_{\alpha+1}=T^+\cap\Omega,\qquad
			U_{\alpha+1}=U_\alpha\dot\cup V^+$
		};
		
		\draw[arr] (init) -- (datum);
		\draw[arr] (datum) -- (small);
		\draw[arr] (small) -- (plus);
		\draw[arr] (plus) -- (next);
		
	\end{tikzpicture}
	\caption{The successor step in the proof of
		Proposition~\ref{prop:small-layer}.}
	\label{fig:small-layer}
\end{figure}
	
	We begin with
	\[
	C_0=R_0=\CC\one,
	\qquad
	\Theta_0=\varnothing,
	\qquad
	U_0=\varnothing.
	\]
	Then \eqref{eq:induction-blocks-new} is immediate, and all the above
	properties hold at $\alpha=0$.
	
	We now treat the successor and limit steps separately.
	
	\medskip
	\noindent
	\emph{Successor step.}
	Suppose $\alpha<\lambda$ and the data have been constructed up to stage
	$\alpha$.  We explain how to define the objects at stage $\alpha+1$.
	
By condition {\rm (d)},
\[
(C_\alpha,R_\alpha,\Theta_\alpha\cup U_\alpha)
\]
is a completed block.
	Since $\Theta_\alpha\subset\Omega$ is $\cl_\Omega$-closed, it is itself a
	bounded background by Proposition~\ref{prop:OA-closure}(iv).  Hence,
	by Definition~\ref{def:background}(ii), the disjoint union $
	\Omega\dot\cup U_\alpha$
	is again a bounded background.  Let $\cl_{\Omega\dot\cup U_\alpha}$
	be a closure operation on this background, chosen compatibly with
	$\cl_\Omega$ as in Proposition~\ref{prop:OA-closure}.
	Here the same prescribed $\cl_\Omega$ is used at every successor.
	The closures on the different enlarged backgrounds may be chosen
	independently: the required compatibility is precisely that the
	old slice of a closed family is closed under this fixed
	$\cl_\Omega$.
	
	Choose $L^2$-dense sets $
	Y_\alpha\subset (C_\alpha)_{\sa},
	X_\alpha\subset (R_\alpha)_{\sa},$
	with $
	|Y_\alpha|, |X_\alpha|\leq\lambda_\alpha.$
Moreover, by the induction hypotheses,
\[
|\Omega\dot\cup U_\alpha|
\leq
\max\{|\Omega|,\lambda_\alpha\}
<
\kappa,
\qquad
E_N(\Omega\dot\cup U_\alpha)=0.
\]
Since $\lambda_\alpha<\lambda<\kappa$, we may apply the induction
hypothesis at the smaller cardinal $\lambda_\alpha$, using the
bounded background $\Omega\dot\cup U_\alpha$ and the following data:
	\[
	S'_\alpha=\Theta_\alpha\cup U_\alpha,
	\qquad
	X'_\alpha=X_\alpha,
	\qquad
	Y'_\alpha=Y_\alpha.
	\]
	We then enlarge one of these three sets according to the tag
	$t_\alpha$:
	\[
	\begin{cases}
		S'_\alpha:=S'_\alpha\cup\{z_\alpha\}, & t_\alpha=0,\\[1mm]
		X'_\alpha:=X'_\alpha\cup\{z_\alpha\}, & t_\alpha\in\{1,3\},\\[1mm]
		Y'_\alpha:=Y'_\alpha\cup\{z_\alpha\}, & t_\alpha=2.
	\end{cases}
	\]
	Each of these sets has cardinality at most $\lambda_\alpha$.
	
	Applying the induction hypothesis at $\lambda_\alpha$ yields von
	Neumann algebras
	\[
	C^+\subset N,
	\qquad
	C^+\subset R^+\subset M,
	\]
	a $\cl_{\Omega\dot\cup U_\alpha}$-closed old-background part
	\[
	T^+\subset\Omega\dot\cup U_\alpha,
	\]
	and an orthonormal family
	\[
	V^+\subset\mathcal S_0(M)
	\]
	such that
	\[
	S'_\alpha\subset T^+,
	\qquad
	X'_\alpha\subset R^+,
	\qquad
	Y'_\alpha\subset C^+,
	\]
	\[
	E_N(V^+)=0,
	\qquad
	V^+\perp(\Omega\dot\cup U_\alpha),
	\]
	\[
	|T^+|,\ |V^+|,\ \dens L^2(R^+)\leq\lambda_\alpha,
	\]
	and
	\begin{equation}\label{eq:successor-block}
		L^2(R^+)_{\sa}
		=
		L^2(C^+)_{\sa}
		\oplus
		\overline{\spanop_{\RR}}(T^+\cup V^+).
	\end{equation}
	\begin{sloppypar}
	Equation~\eqref{eq:successor-block} implies that
	$T^+\cup V^+\subset L^2(R^+)$.  Since these are bounded elements of
	$M$, Lemma~\ref{lem:density-facts}(iii) gives
	$T^+\cup V^+\subset R^+$.  Thus
	$(C^+,R^+,T^+\cup V^+)$ is indeed a completed block.
	\end{sloppypar}
	
Since $Y_\alpha\subset C^+$ and $X_\alpha\subset R^+$
are $L^2$-dense in $(C_\alpha)_{\sa}$ and
$(R_\alpha)_{\sa}$, respectively,
Lemma~\ref{lem:density-facts}(iii) and complex linearity give
\[
C_\alpha\subset C^+,
\qquad
R_\alpha\subset R^+.
\]
Thus the construction is increasing.
	
	Now define
	\[
	C_{\alpha+1}=C^+,
	\qquad
	R_{\alpha+1}=R^+,
	\qquad
	\Theta_{\alpha+1}=T^+\cap\Omega,
	\qquad
	U_{\alpha+1}=U_\alpha\dot\cup V^+.
	\]
	Because $S'_\alpha$ contains $\Theta_\alpha\cup U_\alpha$, we have
	\[
	\Theta_\alpha\cup U_\alpha\subset T^+.
	\]
	Since $T^+\subset\Omega\dot\cup U_\alpha,$
	it follows that
	\[
	T^+=\Theta_{\alpha+1}\dot\cup U_\alpha.
	\]
	Substituting this into \eqref{eq:successor-block}, and recalling that
	$V^+\perp U_\alpha$, we obtain
	\[
	L^2(R_{\alpha+1})_{\sa}
	=
	L^2(C_{\alpha+1})_{\sa}
	\oplus
	\overline{\spanop_{\RR}}(\Theta_{\alpha+1}\cup U_{\alpha+1}).
	\]
	So \eqref{eq:induction-blocks-new} holds at stage $\alpha+1$.
	
	By the compatibility clause in Proposition~\ref{prop:OA-closure},
	$\Theta_{\alpha+1}=T^+\cap\Omega$ is $\cl_\Omega$-closed. Moreover, $U_{\alpha+1}$ is orthonormal because $U_\alpha$ is
	orthonormal and $V^+\perp U_\alpha$.
	Also,
	\[
	E_N(U_{\alpha+1})=0,
	\qquad
	U_{\alpha+1}\perp\Omega,
	\]
	because both properties hold for $U_\alpha$ and for $V^+$, and
	$V^+\perp U_\alpha$.
	Finally,
	\[
	|\Theta_{\alpha+1}|,\ |U_{\alpha+1}|,\ \dens L^2(R_{\alpha+1})
	\leq\lambda_\alpha.
	\]
	Moreover, by construction the new tagged datum $(t_\alpha,z_\alpha)$ is
	absorbed at stage $\alpha+1$, so condition (e) continues to hold.
	
	\medskip
	\noindent
	\emph{Limit step.}
	Let $\delta\leq\lambda$ be a nonzero limit ordinal, and suppose that the data
	have been constructed for all $\alpha<\delta$.  Define
	\[
	C_\delta
	=
	W^*\!\left(\bigcup_{\alpha<\delta}C_\alpha\right),
	\qquad
	R_\delta
	=
	W^*\!\left(\bigcup_{\alpha<\delta}R_\alpha\right),
	\]
	\[
	\Theta_\delta
	=
	\bigcup_{\alpha<\delta}\Theta_\alpha,
	\qquad
	U_\delta
	=
	\bigcup_{\alpha<\delta}U_\alpha.
	\]The family $U_\delta$ is orthonormal, being an increasing union of
	orthonormal families.
	
Since the family $(\Theta_\alpha)_{\alpha<\delta}$ is increasing and
	each $\Theta_\alpha$ is $\cl_\Omega$-closed,
	Proposition~\ref{prop:OA-closure}(iii) gives
	\[
	\begin{aligned}
		\cl_\Omega(\Theta_\delta)
		&=
		\cl_\Omega\left(\bigcup_{\alpha<\delta}\Theta_\alpha\right)\\
		&=
		\bigcup_{\alpha<\delta}\cl_\Omega(\Theta_\alpha)\\
		&=
		\bigcup_{\alpha<\delta}\Theta_\alpha\\
		&=
		\Theta_\delta.
	\end{aligned}
	\]
	So $\Theta_\delta$ is $\cl_\Omega$-closed.
	
Put $\mu_\delta=\max\{\aleph_0,|\delta|\}$. For every
$\alpha<\delta$, we have $\lambda_\alpha\leq\mu_\delta$,
and $|\delta|\leq\mu_\delta$. Hence any family
$(F_\alpha)_{\alpha<\delta}$ with $|F_\alpha|\leq\lambda_\alpha$
satisfies
\[
\left|\bigcup_{\alpha<\delta}F_\alpha\right|
\leq\sum_{\alpha<\delta}|F_\alpha|
\leq|\delta|\cdot\mu_\delta
\leq\mu_\delta^2=\mu_\delta.
\]
The same common bound applies to the dense sets chosen below.
For $\delta<\lambda$ it is strictly below $\lambda$, and at
$\delta=\lambda$ it equals $\lambda$.
Applying it to $\Theta_\alpha$ and $U_\alpha$ gives
\[
\left|\bigcup_{\alpha<\delta}\Theta_\alpha\right|,
\quad
\left|\bigcup_{\alpha<\delta}U_\alpha\right|
\leq
\max\{\aleph_0,|\delta|\}.
\]
For each $\alpha<\delta$, choose an $L^2$-dense set
$D_\alpha\subset R_\alpha$ such that
$|D_\alpha|\leq\lambda_\alpha$.  By
Lemma~\ref{lem:density-facts}(ii),
\[
\bigcup_{\alpha<\delta}D_\alpha
\]
is $L^2$-dense in $R_\delta$.  Moreover,
\[
\left|\bigcup_{\alpha<\delta}D_\alpha\right|
\leq
\max\{\aleph_0,|\delta|\}.
\]
Consequently,
\[
|\Theta_\delta|,\ |U_\delta|,\ \dens L^2(R_\delta)
\leq
\max\{\aleph_0,|\delta|\}.
\]
	Also,
	\[
	E_N(U_\delta)=0,
	\qquad
	U_\delta\perp\Omega,
	\]
	because these properties hold at every earlier stage.
	
	We next verify \eqref{eq:induction-blocks-new} at stage $\delta$.
	The two spaces
	\[
	L^2(C_\delta)_{\sa}
	\qquad\text{and}\qquad
	\overline{\spanop_{\RR}}(\Theta_\delta\cup U_\delta)
	\]
	are orthogonal and both lie in $L^2(R_\delta)_{\sa}$.
	For orthogonality, take a vector in $L^2(C_\alpha)_{\sa}$ and a
	member of $\Theta_\beta\cup U_\beta$. They belong to a common later
	stage, whose completed-block decomposition makes them orthogonal.
	Passing to real closed spans and $L^2$-limits proves the assertion
	for the two displayed spaces.
	Conversely, Lemma~\ref{lem:density-facts}(ii), followed by
	self-adjointization, gives
	\[
	L^2(R_\delta)_{\sa}
	=
	\overline{\bigcup_{\alpha<\delta}L^2(R_\alpha)_{\sa}}^{\,\|\cdot\|_2}.
	\]
	Indeed, involution preserves each stage and is an $L^2$-isometry,
	so taking real parts preserves convergence to a self-adjoint vector.
	Each vector in a stage on the right has the decomposition
	\eqref{eq:induction-blocks-new} at that stage. Hence
	\[
	L^2(R_\delta)_{\sa}
	\subset
	L^2(C_\delta)_{\sa}
	+
	\overline{\spanop_{\RR}}(\Theta_\delta\cup U_\delta).
	\]
	Since the sum on the right is orthogonal, it is closed, and therefore
	\[
	L^2(R_\delta)_{\sa}
	=
	L^2(C_\delta)_{\sa}
	\oplus
	\overline{\spanop_{\RR}}(\Theta_\delta\cup U_\delta).
	\]
	So \eqref{eq:induction-blocks-new} also holds at limit stages.
	
	Finally, let
	\[
	C=C_\lambda,
	\qquad
	R=R_\lambda,
	\qquad
	\Omega_0=\Theta_\lambda,
	\qquad
	U=U_\lambda.
	\]
	By construction,
	\[
	S\subset\Omega_0,
	\qquad
	X\subset R,
	\qquad
	Y\subset C,
	\]
	and
	\[
	E_N(U)=0,
	\qquad
	U\perp\Omega.
	\]
	Moreover,
	\[
	|\Omega_0|,\ |U|,\ \dens L^2(R)\leq\lambda,
	\]
	and \eqref{eq:induction-blocks-new} at stage $\lambda$ is exactly
	\eqref{eq:small-block}.  This completes the proof.
	
	Observe that no regularity assumption on $\lambda$ was used.
\end{proof}

For comparison, \cite[p.~1389, Theorem~3.4(ii)]{DM23} gives an
extension result for uniformly bounded self-adjoint orthonormal bases
arising from diffuse subalgebras in the separable setting. That result
does not require the added basis elements to be unitaries. The proposition
below preserves a given symmetry basis, with the ambient central density
condition providing the room needed for the extension. We extract the
form used in the proof of the main theorem.
	
	\begin{prop}\label{prop:relative-extension}
Let $(M,\tau)$ be a finite tracial von Neumann algebra, set
$\kappa=\dens L^2(M,\tau)$, and assume that $\kappa>\aleph_0$
and that $(M,\tau)$ satisfies \eqref{eq:CD}.
Let $N\subset M$ be a von Neumann subalgebra, and let $\lambda$
be an infinite cardinal such that
\[
\lambda<\kappa,\qquad \dens L^2(N)\leq\lambda.
\]
Suppose $H_0(N)$ has an orthonormal basis
$B_N\subset\mathcal S_0(N)$. If $X\subset M_{\sa}$ and
$|X|\leq\lambda$, then there are a von Neumann algebra
$N\cup X\subset P\subset M$ and an orthonormal family
$U\subset\mathcal S_0(P)$ such that
\[
|U|\leq\lambda,\qquad
\dens L^2(P)\leq\lambda,\qquad
H_0(P)=H_0(N)\oplus\overline{\spanop_{\RR}}U.
\]
Consequently $B_N\cup U$ is an orthonormal basis of $H_0(P)$.
\end{prop}

\begin{proof}
We first choose a self-adjoint generating set $Y\subset N_{\sa}$
with $|Y|\leq\lambda$ and $W^*(Y)=N$.
To see that this is possible, approximate a dense subset of $L^2(N)$
of cardinal at most $\lambda$ by elements of $N$ at countably many
accuracies. This gives an $L^2$-dense set $D\subset N$ with
$|D|\leq\lambda$. Take the real and imaginary parts of the elements
of $D$ to form $Y$. Then $D\subset W^*(Y)\subset N$, so
Lemma~\ref{lem:density-facts}(iii) gives $W^*(Y)=N$.

Apply Proposition~\ref{prop:small-layer} with empty background,
the prescribed set $X$, and this set $Y$, and put $P=R$ for its
upper algebra. It gives
\[
L^2(P)_{\sa}
=
L^2(C)_{\sa}\oplus\overline{\spanop_{\RR}}U,
\qquad
Y\subset C\subset N,
\]
together with $X\subset P$ and
$|U|,\dens L^2(P)\leq\lambda$.
Since $W^*(Y)=N$, we have $C=N$.
Also $\widehat u\in L^2(P)$ for each $u\in U$, so
Lemma~\ref{lem:density-facts}(iii) gives $U\subset P$.

Here $L^2(N,\tau|_N)$ and $L^2(P,\tau|_P)$ are identified with
their closed images in $L^2(M,\tau)$ under the canonical
isometries $\widehat a\mapsto\widehat a$. These isometries preserve
the self-adjoint real parts and the vector $\widehat{\one}$,
and therefore identify the corresponding $H_0$ spaces.
Since every $u\in U$ has trace zero, comparison with
\[
L^2(N)_{\sa}=\RR\widehat{\one}\oplus H_0(N),
\qquad
L^2(P)_{\sa}=\RR\widehat{\one}\oplus H_0(P)
\]
yields
\[
H_0(P)=H_0(N)\oplus\overline{\spanop_{\RR}}U.
\]
Thus the literal union $B_N\cup U$ is a basis, with every old
operator representative unchanged under inclusion.
The construction of $P$ and $U$ did not use the given basis of
$H_0(N)$ and does not require $N$ to be diffuse.
\end{proof}
	
	\section{Proof of Theorem~\ref{thm:nonseparable-main}}\label{sec:main-proof}
	\subsection{Walsh symmetries}
	
	We record an explicit diffuse abelian construction. It provides a symmetry
basis in its own tracial $L^2$-space; the main transfinite construction below
will use the simpler scalar starting algebra instead.
	
	\begin{lem}\label{lem:starting-algebra}
		Every diffuse finite tracial von Neumann algebra $(M,\tau)$ contains an
		$L^2$-separable diffuse abelian von Neumann subalgebra $A_0$ for which
		$H_0(A_0)$ has an orthonormal basis in $\mathcal S_0(A_0)$.
	\end{lem}
	
	\begin{proof}
We need only equal-trace bisection of projections. Here is a direct
finite-constraint argument. For a projection $p\neq0$, the set
\[
K_p=\{0\leq a\leq\one:
\tau(a)=\tfrac12\tau(p),\quad \tau((\one-p)a)=0\}
\]
is nonempty, ultraweakly compact and convex; it contains $p/2$.
Choose an extreme point $q$. If $q$ is not a projection, the
spectral-corner argument in the proof of Lemma~\ref{lem:relative-norming}
gives a nonzero projection $f$ commuting with $q$ and a number
$\varepsilon>0$ such that
\[
\varepsilon f\leq fqf\leq(1-\varepsilon)f.
\]
This spectral-corner assertion concerns a single positive contraction
and does not use~\eqref{eq:CD}. Diffuseness permits us to choose three
nonzero mutually orthogonal projections below $f$. Their real span has
dimension three, whereas the map
\[
h\longmapsto\bigl(\tau(h),\tau((\one-p)h)\bigr)
\]
has values in $\RR^2$ on self-adjoint elements. Thus it has a nonzero
self-adjoint kernel vector $h\in fMf$. Rescale $h$ so that
$\|h\|_\infty<\varepsilon$. The corner inequalities give
$0\leq q\pm h\leq\one$, and both trace constraints are preserved.
This contradicts extremality. Hence $q$ is a projection. Traciality gives
\[
\| (\one-p)q\|_2^2
=\tau(q(\one-p)q)=\tau((\one-p)q)=0.
\]
Faithfulness implies $(\one-p)q=0$, so $q\leq p$ and
$\tau(q)=\tau(p-q)=\tau(p)/2$.

Put $p_\varnothing=\one$ and recursively bisect each $p_\sigma$ in
this way. For each $n\geq0$, the projections indexed by words of
length $n$ are mutually orthogonal and sum to $\one$, and
\[
p_\sigma=p_{\sigma0}+p_{\sigma1},
\qquad \tau(p_\sigma)=2^{-|\sigma|}.
\]
Projections at different levels commute: either one word extends the
other, giving containment, or they lie below different members of a
common partition, giving orthogonality. Consequently
\[
D_n=\spanop_{\CC}\{p_\sigma:|\sigma|=n\},
\qquad A_0=W^*\left(\bigcup_{n\geq0}D_n\right)
\]
are abelian von Neumann algebras, and $D_n\subset D_{n+1}$.
The algebra $A_0$ is countably generated, so
Lemma~\ref{lem:density-facts}(i) makes its $L^2$-space separable.
If $0\neq r\in A_0$ were a minimal projection, then at each level
exactly one $p_\sigma$ would contain $r$: the products $rp_\sigma$
are orthogonal subprojections of $r$ summing to $r$.
This would imply $0<\tau(r)\leq2^{-n}$ for every $n$, a contradiction.
Thus $A_0$ is diffuse.

Index coordinates by the positive integers. For a finite set
$F\subset\{1,2,\ldots\}$, choose $n$ with $F\subset\{1,\ldots,n\}$
and define
\[
w_F=\sum_{\sigma\in\{0,1\}^n}
(-1)^{\sum_{j\in F}\sigma_j}p_\sigma.
\]
Refinement leaves this operator unchanged, because the coefficient is
the same on the two children of each word. Orthogonality of the row
projections gives $w_F=w_F^*$ and $w_F^2=\one$. For
$F,G\subset\{1,\ldots,n\}$,
\[
\langle w_F,w_G\rangle_2
=2^{-n}\sum_{\sigma\in\{0,1\}^n}
(-1)^{\sum_{j\in F\mathbin\triangle G}\sigma_j}
=\begin{cases}1,&F=G,\\0,&F\neq G.\end{cases}
\]
For a nonempty symmetric difference, pairing each word with the word
obtained by flipping one coordinate in that difference proves the
vanishing. The inverse finite character formula is
\[
p_\sigma=2^{-n}\sum_{F\subset\{1,\ldots,n\}}
(-1)^{\sum_{j\in F}\sigma_j}w_F.
\]
All coefficients are real. Since the nonzero row projections form a
basis of $D_n$ and a self-adjoint element has real coefficients in this
basis, the $w_F$ at level $n$ span exactly $L^2(D_n)_{\sa}$ over $\RR$.

By Lemma~\ref{lem:density-facts}(ii), $\bigcup_n D_n$ is $L^2$-dense
in $A_0$. Taking $a\mapsto(a+a^*)/2$ shows that the corresponding
self-adjoint union is dense in $L^2(A_0)_{\sa}$. Since, for each $n$,
the family
$
\{w_F:F\subset\{1,\ldots,n\}\}
$
is an orthonormal basis of $L^2(D_n)_{\sa}$, and since $w_F$ is
unchanged upon passing to a finer dyadic level, it follows that
$
\{w_F:F\subset\mathbb N,\ F\text{ finite}\}
$
is an orthonormal basis of $L^2(A_0)_{\sa}$. These operators are the
dyadic Walsh functions; see~\cite[pp.~8--10]{Wal23}.
 Since $w_\varnothing=\one$ and
$\tau(w_F)=0$ for $F\neq\varnothing$, deleting the identity gives
\[
\{w_F:\varnothing\neq F\subset\{1,2,\ldots\},\ F\text{ finite}\}
\]
as an orthonormal basis of $H_0(A_0)$. Here $L^2(A_0)$ denotes the
completion for $\tau|_{A_0}$; its canonical isometric inclusion into
$L^2(M)$ preserves the displayed bounded representatives and the
orthogonal complement of $\one$.
\end{proof}
	\subsection{Proof of the main theorem}
	
	We now prove the main theorem.
	
	\begin{proof}[Proof of Theorem~\ref{thm:nonseparable-main}]
The implications (iii)$\Rightarrow$(ii)$\Rightarrow$(i) are immediate,
and (i)$\Rightarrow$(iv) is
Proposition~\ref{prop:central-density-necessary}. Assume \eqref{eq:CD}.

If $L^2(M)$ is separable, condition~(iii) follows from Theorem~1
and its proof in \cite{HTZ26}: that proof starts from
$E_0=\spanop_{\RR}\{\one\}$, so all the remaining basis vectors
are trace-zero symmetries. We may therefore assume that
$\kappa=\dens L^2(M)>\aleph_0$.

As in the proof of Proposition~\ref{prop:relative-extension},
choose $G\subset M_{\sa}$ with $|G|\leq\kappa$ and $W^*(G)=M$.
Adjoining $0$ if necessary, take $G$ nonempty.
Let $\theta$ be the initial ordinal of cardinality $\kappa$,
and choose a family $(g_\alpha)_{\alpha<\theta}$ with range $G$,
allowing repetitions. For $\alpha\leq\theta$ put
\[
\mu_\alpha=\max\{\aleph_0,|\alpha|\}.
\]
Then $\mu_\alpha<\kappa$ for $\alpha<\theta$, whereas
$\mu_\theta=\kappa$.

We construct increasing von Neumann algebras $N_\alpha\subset M$
and nested families $B_\alpha\subset\mathcal S_0(N_\alpha)$,
for $\alpha\leq\theta$, such that
\[
B_\alpha\text{ is an orthonormal basis of }H_0(N_\alpha),
\qquad
\dens L^2(N_\alpha)\leq\mu_\alpha,
\]
and $g_\beta\in N_\alpha$ whenever $\beta<\alpha$.
Start with
\[
N_0=\CC\one,\qquad B_0=\varnothing.
\]
Indeed, $H_0(\CC\one)=\{0\}$ and
$\dens L^2(\CC\one)\leq\aleph_0$.
Equivalently, each stage is a completed block over the same scalar
base, with
\[
L^2(N_\alpha)_{\sa}
=
\RR\widehat{\one}\oplus\overline{\spanop_{\RR}}B_\alpha.
\]

At a successor stage $\alpha+1\leq\theta$, apply
Proposition~\ref{prop:relative-extension} to $N_\alpha$,
$B_\alpha$, $X=\{g_\alpha\}$, and $\lambda=\mu_\alpha<\kappa$.
It gives $N_{\alpha+1}\supset N_\alpha\cup\{g_\alpha\}$ and
$B_{\alpha+1}=B_\alpha\cup U_\alpha$, retaining the old
representatives, with
$\dens L^2(N_{\alpha+1})\leq\mu_\alpha\leq\mu_{\alpha+1}$.
The two orthogonal decompositions compose over the common algebra
$N_\alpha$, so the scalar-base completed block is also retained.

At a nonzero limit ordinal $\delta\leq\theta$, set
\[
N_\delta=W^*\left(\bigcup_{\alpha<\delta}N_\alpha\right),
\qquad
B_\delta=\bigcup_{\alpha<\delta}B_\alpha.
\]
Any two elements of $B_\delta$ belong to a common earlier stage,
so $B_\delta$ is orthonormal.
Lemma~\ref{lem:density-facts}(ii), followed by taking self-adjoint
parts and subtracting their traces, gives
\[
H_0(N_\delta)
=
\overline{\bigcup_{\alpha<\delta}H_0(N_\alpha)}^{\|\cdot\|_2}
=
\overline{\spanop_{\RR}}B_\delta.
\]
The absorption condition also persists at the limit.

For the density estimate, all earlier stages have the single
common bound $\dens L^2(N_\alpha)\leq\mu_\delta$.
Choose self-adjoint generating sets $G_\alpha$ for $N_\alpha$
with $|G_\alpha|\leq\mu_\delta$, as above. Their union generates
$N_\delta$ and satisfies
\[
\left|\bigcup_{\alpha<\delta}G_\alpha\right|
\leq |\delta|\cdot\mu_\delta
\leq\mu_\delta\cdot\mu_\delta
=\mu_\delta.
\]
Lemma~\ref{lem:density-facts}(i) therefore gives
$\dens L^2(N_\delta)\leq\mu_\delta$.
For $\delta<\theta$ this bound is strictly below $\kappa$;
at the final stage $\delta=\theta$ the bound $\kappa$ is allowed.
This uses a fixed bound at each limit and requires no regularity
of $\kappa$.

Every $g_\alpha$ is absorbed at stage $\alpha+1$, so
$G\subset N_\theta$. Consequently
\[
M=W^*(G)\subset N_\theta\subset M,
\]
and hence $N_\theta=M$.
Thus $B_\theta\subset\mathcal S_0(M)$ is an orthonormal basis of
$H_0(M)$. Lemma~\ref{lem:trace-zero-reduction} proves (iii),
and Lemma~\ref{lem:trace-vectors-unitaries} gives the trace-vector
formulation.
\end{proof}

\end{document}